\pdfoutput=1
\documentclass[a4paper,reqno]{amsart}

\usepackage{amsthm}
\usepackage{amssymb}
\usepackage{dsfont}
\usepackage{enumitem}
\usepackage{stmaryrd}
\usepackage[mathcal]{euscript}
\usepackage{tikz-cd} 
\usepackage{microtype}
\usepackage{mathtools}

\usepackage[top=36mm, bottom=36mm, outer=30mm, inner=30mm, marginparwidth=20mm]{geometry}

\usepackage[pdfusetitle]{hyperref}
\hypersetup{%
  bookmarksnumbered=true,%
  colorlinks=true,%
  linkcolor=[RGB]{0, 0, 153},
  citecolor=[RGB]{0,153,51},      
  urlcolor=[RGB]{153,0,0},
  filecolor=blue,%
  menucolor=blue,%
  bookmarksopen=true,%
  bookmarksdepth=2,%
  pageanchor=true,
  breaklinks=true}

\usepackage[nameinlink,capitalize]{cleveref}

\usepackage[
sorting=nyt,
backend=bibtex,
style=alphabetic,
url=false,
doi=false,
isbn=false, %
maxnames=50
]{biblatex}
\renewbibmacro{in:}{}
\defbibheading{bibliography}[\refname]{\section*{#1}}%
\makeatletter
\@namedef{subjclassname@2020}{\textup{2020} Mathematics Subject Classification}
\newcommand{\itemcase}[2]{\item[\textup{(#1)}]\def\@currentlabel{\textup{(#1)}}\label{#2}}
\makeatother

\theoremstyle{plain}
\newtheorem{thm-intro}{Theorem}
\expandafter\renewcommand\csname thethm-intro\endcsname{\Alph{thm-intro}}
\crefname{thm-intro}{Theorem}{Theorems}

\usepackage{aliascnt}
\newcounter{thm}[section]

\theoremstyle{plain}
\newaliascnt{theorem}{thm}
\newtheorem{theorem}[theorem]{Theorem}
\aliascntresetthe{theorem}

\newaliascnt{proposition}{thm}
\newtheorem{proposition}[proposition]{Proposition}
\aliascntresetthe{proposition}

\newaliascnt{defproposition}{thm}
\newtheorem{defproposition}[defproposition]{Definition-Proposition}
\aliascntresetthe{defproposition}

\newaliascnt{lemma}{thm}
\newtheorem{lemma}[lemma]{Lemma}
\aliascntresetthe{lemma}

\newaliascnt{corollary}{thm}

\aliascntresetthe{corollary}

\newaliascnt{conjecture}{thm}

\aliascntresetthe{conjecture}

\theoremstyle{definition}
\newaliascnt{definition}{thm}
\newtheorem{definition}[definition]{Definition}
\aliascntresetthe{definition}

\newaliascnt{example}{thm}
\newtheorem{example}[example]{Example}
\aliascntresetthe{example}

\newaliascnt{construction}{thm}

\aliascntresetthe{construction}

\newaliascnt{setup}{thm}
\newtheorem{setup}[setup]{Setup}
\aliascntresetthe{setup}

\newaliascnt{notation}{thm}

\aliascntresetthe{notation}

\newaliascnt{chunk}{thm}

\aliascntresetthe{chunk}

\theoremstyle{remark}
\newaliascnt{remark}{thm}
\newtheorem{remark}[remark]{Remark}
\aliascntresetthe{remark}

\newaliascnt{question}{thm}

\aliascntresetthe{question}

\theoremstyle{remark}
\newtheorem*{ack}{Acknowledgments}
\newtheorem*{AI}{AI disclosure}
\newtheorem*{conventions}{Conventions}

\newtheorem*{orga-proof}{Organization of the paper and strategy of proof}

\crefformat{equation}{(#2#1#3)}
\crefname{theorem}{Theorem}{Theorems}
\crefname{proposition}{Proposition}{Propositions}
\crefname{defproposition}{Definition-Proposition}{Definition-Propositions}
\crefname{lemma}{Lemma}{Lemmas}
\crefname{corollary}{Corollary}{Corollaries}
\crefname{conjecture}{Conjecture}{Conjectures}
\crefname{definition}{Definition}{Definitions}
\crefname{example}{Example}{Examples}
\crefname{construction}{Construction}{Constructions}
\crefname{setup}{Setup}{Setups}
\crefname{notation}{Notation}{Notations}
\crefname{remark}{Remark}{Remarks}
\crefname{question}{Question}{Questions}

\DeclareMathOperator{\Spec}{Spec}

\DeclareMathOperator{\Tr}{Tr}

\DeclareMathOperator{\Hom}{Hom}

\newcommand{\pet}{\pi_1^{\mathrm{\acute et}}}
\newcommand{\piN}{\pi_1^{\mathrm{N}}}
\newcommand{\PP}{\mathbb{P}}
\newcommand{\ZZ}{\mathbb{Z}}

\newcommand{\Coh}{\operatorname{Coh}}

\renewcommand{\dim}{\operatorname{dim}}

\newcommand{\id}{\mathrm{id}}

\newcommand{\Pic}{\operatorname{Pic}}

\newcommand{\reg}{\operatorname{reg}}

\newcommand{\bbF}{\mathbb F}

\newcommand{\bbQ}{\mathbb Q} 
 
\newcommand{\bbZ}{\mathbb Z}

\newcommand{\fm}{\mathfrak{m}}

\newcommand{\cO}{\mathcal{O}}
\newcommand{\cF}{\mathcal{F}}
\newcommand{\cL}{\mathcal{L}}

\newcommand{\cHom}{\mathcal{H}\!om}
\newcommand{\Gm}{\mathbb{G}_{m}}

\newcommand{\Qcoh}{\operatorname{QCoh}}

\newcommand{\isom}{\cong}
\newcommand{\BG}{\mathrm{B}}      %

\renewcommand{\le}{\leqslant}
\renewcommand{\ge}{\geqslant}

\iffalse

\newcommand{\Charles}[1]{}

\else

\newcommand{\marg}[1]{\normalsize{{
			\color{red}\footnote{{\color{blue}#1}}}{\marginpar[\vskip
			-.25cm{\color{red}\hfill\thefootnote$\implies$}]{\vskip
				-.2cm{\color{red}$\impliedby$\tiny\thefootnote}}}}}
\newcommand{\Charles}[1]{\marg{(Charles) #1}}

\fi

\title[Fundamental group schemes of globally $F$-regular varieties]{
	Fundamental group schemes\\ of globally $F$-regular varieties
}

\author[Charles~Vial]{Charles Vial}
\address{Fakult\"at f\"ur Mathematik, Universit\"at Bielefeld, D-33501 Bielefeld, Germany
}
\email{vial@math.uni-bielefeld.de}

\thanks{This research was funded by the Deutsche Forschungsgemeinschaft (DFG, German Research Foundation) -- Project-ID~491392403 -- TRR~358.}

\begin{document}

\begin{abstract} 
Let $X$ be a globally $F$-regular projective variety over an algebraically closed field
of positive characteristic. We show that the Nori fundamental group scheme of its regular locus
is finite and linearly reductive. Finiteness is the counterpart of theorems of Xu and Braun on
the fundamental group of the smooth locus of a complex log Fano variety. 
The key input is that if a cover of a splinter is a torsor in codimension one and induces an isomorphism on global functions, then its structure group is linearly reductive.
We then show that towers of such covers stabilize,
both over globally $F$-regular varieties and, in mixed characteristic, over the globally $+$-regular
schemes considered by Cai--Lee--Ma--Schwede--Tucker.
\end{abstract}


\date{\today}

\maketitle

\vspace{-10pt}
\section{Introduction}

Varieties of Fano type and Kawamata log terminal (klt) singularities are among
the basic objects produced by the minimal model program, and a natural measure
of their topological complexity is the fundamental group. Over the complex numbers, this
fundamental group is finite, both globally and locally. 
Recall that a normal projective variety $X$ over a field is
\emph{log Fano} if there is an effective $\bbQ$-divisor $\Delta$ such that the
pair $(X,\Delta)$ is klt and $-(K_X+\Delta)$ is ample.
Globally, Takayama~\cite{takayama} showed that the fundamental group of a log Fano complex projective variety $X$ is trivial and
Xu~\cite{xu_finite_fundamental_group_klt} proved that the \'etale fundamental group
of its regular locus $X^{\reg}$ is finite.
Locally,
Xu also proved that the local \'etale fundamental group of a complex klt singularity
is finite. 
Braun \cite{braun_finite_fundamental_group_klt}
strengthened both statements of Xu to the topological fundamental group, thereby
confirming Koll\'ar's conjecture on the finiteness of the local fundamental group
of a klt singularity. At the level of covers, Greb--Kebekus--Peternell
\cite{greb_kebekus_peternell_etale_fundamental_groups_of_klt_spaces}, building on
Xu's work, proved that every tower of normal, quasi-\'etale covers, that is, finite
covers that are \'etale in codimension one, of a complex quasi-projective klt
variety is eventually \'etale, provided the composed covers are Galois. 

Let now $k$ be a field of characteristic $p>0$ with algebraic closure $\bar k$. 
In positive characteristic, strongly $F$-regular singularities~\cite{hochster_huneke_tight_closure_and_strong_f_regularity} play the role of klt
singularities \cite{hara_watanabe_f_regular_f_pure_vs_log_terminal_log_canonical}, and globally
$F$-regular varieties \cite{schwede_smith_globally_f_regular_and_log_Fano_varieties} that of log Fano varieties.
 For
\'etale covers, much of the complex picture has been reproduced, by local methods.
Carvajal-Rojas--Schwede--Tucker
\cite{carvajal-rojas_schwede_tucker_fundamental_groups_of_f_regular_singularities_via_f_signature}
proved that the \'etale fundamental group of the regular locus of an $F$-finite, strongly
$F$-regular, strictly henselian local domain $R$ over $k$ of dimension at least~$2$ is finite of order
prime to~$p$
and deduced, by passing to cones, that $\pet(X^{\reg})$ is finite of order prime to~$p$ for $X$ globally
$F$-regular projective over~$\bar k$\,;
and Bhatt--Carvajal-Rojas--Graf--Schwede--Tucker \cite{bhatt_carvajal-rojas_graf_schwede_tucker}
established the characteristic-$p$ analogue of the tower theorem of Greb--Kebekus--Peternell for
$F$-finite strongly $F$-regular schemes.
However, the natural covers in positive characteristic are the torsors under arbitrary finite group
schemes, as encoded in Nori's fundamental group scheme \cite{nori_the_fundamental_group_scheme}, and torsors under infinitesimal group
schemes, such as $\mu_p$ and $\alpha_p$, are invisible to the \'etale fundamental group. 
 Locally, Carvajal-Rojas
\cite[Main~Thm.]{carvajal_rojas_finite_torsors_over_strongly_f_regular_singularities} proved an
extension theorem\,: if $k$ is algebraically closed, every $R$ as above with residue field $k$ 
admits a finite tower of local quasi-torsor covers
$R\subseteq R_1\subseteq\cdots\subseteq R_n=R^\star$ by strongly $F$-regular local domains such
that, for every finite $k$-group scheme $G$ whose neutral component is solvable, every
$G$-torsor over the regular locus of~$\Spec R^\star$ extends to a $G$-torsor over~$\Spec R^\star$. 
 Globally,
 with Krah we established the positive-characteristic analogue of Takayama's theorem~\cite{takayama} by
showing \cite[Thm.~A]{krah-vial} that 
the Nori fundamental group
scheme of a connected proper splinter (e.g.\ of a globally $F$-regular projective variety\,; see \S \ref{sec:splinters}) over $\bar k$ is trivial. 

In mixed characteristic, Cai--Lee--Ma--Schwede--Tucker~\cite[Thm.~7.4]{cai-lee-ma-schwede-tucker}
established the finiteness of the 
\'etale fundamental group of the regular locus of suitable \emph{globally $+$-regular} projective schemes over
a complete local domain of mixed characteristic with
algebraically closed residue field (see \cref{setup:unif}\ref{setup:unif-0}). 
The key step \cite[Thm.~8.2]{krah-vial} of \cite[Thm.~A]{krah-vial} extends to this setting\,: by \cref{thm:KV-mixed}, if $X$ is a proper
splinter over a Noetherian ring~$R$, and if the closed points of~$X$ have residue fields of
positive characteristic, then every finite \'etale cover or torsor $\pi\colon Y\to X$ under a finite
locally free $R$-group scheme with $H^0(X,\cO_X)=H^0(Y,\cO_Y)$ is an isomorphism.
\medskip

The aim of this work is to establish finiteness properties of torsors under finite group schemes
over the regular loci of globally $F$-regular projective varieties in positive characteristic and of
suitable globally $+$-regular projective schemes in mixed characteristic as in
\cite{cai-lee-ma-schwede-tucker}. A key step will consist in first establishing that suitable such
torsors are torsors under \emph{linearly reductive group schemes}. Recall that
a finite locally free group scheme~$G$ over a scheme~$S$ is linearly reductive if $V\mapsto V^G$ is
exact on quasi-coherent $G$-equivariant $\cO_S$-modules
\cite[Def.~2.2]{abramovich_olsson_vistoli_tame_stacks}\,; if $S$ is the spectrum of a field, this is
equivalent to every finite-dimensional representation of~$G$ being a direct sum of irreducible
representations \cite[p.~1064]{abramovich_olsson_vistoli_tame_stacks}.
\medskip

We start with the positive characteristic situation for which we may formulate our results in terms of the Nori fundamental group scheme.
If $X$ is a normal proper variety over $k$ whose regular locus $X^{\reg}$ admits a $k$-point  $x\in X^{\reg}(k)$, then $X^{\reg}$ has a profinite Nori
fundamental group scheme $\piN(X^{\reg},x)$
\cite{nori_the_fundamental_group_scheme,borne_vistoli_nori_fundamental_gerbe}, whose finite
quotients are
realized by pointed torsors $V\to X^{\reg}$ under finite group schemes with
$H^0(V,\cO_V)=k$\,; see~\cref{lem:nori-torsor}.
Our first
result is the counterpart of Xu's global finiteness
theorem \cite{xu_finite_fundamental_group_klt}. 

\begin{thm-intro}
	\label{thm:piN}
	Let $X$ be a globally $F$-regular projective variety over~$k$ and let
	$x\in X^{\reg}(k)$.
	Then $\piN(X^{\reg},x)$ is a finite linearly reductive $k$-group scheme.
\end{thm-intro}

We refer to \cref{thm:piN2}, where $X^{\reg}$ is more generally replaced by any \emph{big open} subset $U$ of~$X$, i.e.\ any open subset
whose complement has codimension at least~$2$, and where in addition it is shown that the order of $\piN(U,x)$ is bounded by the inverse of the $F$-signature of the local ring at the vertex of the affine cone
over $X\times_k\bar k$ defined by any ample invertible sheaf.
If $k$ is algebraically closed, the \'etale part of
\cref{thm:piN} contains the finiteness of~$\pet(X^{\reg})$, including the fact that its order is prime
to~$p$, of
\cite[Cor.~5.7]{carvajal-rojas_schwede_tucker_fundamental_groups_of_f_regular_singularities_via_f_signature}.
The new content of \cref{thm:piN} thus lies in the infinitesimal part of $\piN(X^{\reg},x)$, which is
invisible to the \'etale fundamental group.
That infinitesimal torsors genuinely occur is illustrated by the following basic example\,:
for $p$ odd, the quotient $X$ of $\PP^2_{\bar k}$ by the action
$t\cdot[z_0:z_1:z_2]=[z_0:tz_1:t^2z_2]$ of $\mu_p$ is a toric, hence globally $F$-regular
\cite{smith_globally_f_regular_varieties_applications_to_vanishing_theorems_for_quotients_of_fano_varieties},
projective surface, and the quotient map restricts to a nontrivial $\mu_p$-torsor over~$X^{\reg}$, with total space the complement of the three coordinate points in~$\PP^2$.
The proof of \Cref{thm:piN} consists of two steps.
 First, we show that the finite quotients of $\piN(X^{\reg},x)$
are linearly reductive.
 Second, linear reductivity
allows us to bound the order of $\piN(X^{\reg},x)$.
\medskip

Both steps hold in fact over a more general base, as we set out to explain now.
Let $S$ be a scheme and let $X$ be a
 locally Noetherian $S$-scheme.
 A  \emph{quasi-torsor} over $X$ under a finite locally free $S$-group scheme $G$  is a morphism $\pi\colon Y\to X$ together with a
 $G$-torsor structure on its restriction to a big open subset of~$X$.
If $X$ is in addition integral and normal, 
  a
\emph{quasi-torsor cover} of $X$ under~$G$ is a finite surjective quasi-torsor $\pi\colon Y\to X$ under $G$ with
$Y$ an $S_2$ scheme whose irreducible components all dominate~$X$, see \cref{def:quasitorsor}.
 If $X^{\reg}$ is open, e.g.\ $X$ excellent, then by purity,
the quasi-torsor covers of $X$ under $G$ correspond to the $G$-torsors over~$X^{\reg}$,
see \cref{prop:quasitorsor-correspondence}. 

\begin{thm-intro}[see \cref{thm:tame-base} and \cref{prop:qtorsor-splinter}]\label{intro:tame}
	Let $R$ be a ring, let $h\colon X\to\Spec R$ be an $R$-scheme, and let $\pi\colon Y\to X$ be a quasi-torsor cover under a finite locally free
	$R$-group scheme~$G$ such that $\pi^\sharp\colon H^0(X,\cO_X)\to H^0(Y,\cO_Y)$ is an
	isomorphism. Assume that $X$ is a splinter.
	Then $G\times_RX$ is linearly reductive over~$X$, and $G$ is linearly reductive
	over~$R$ if $h$ is surjective. Moreover, $Y$ is a splinter, and if $R=k$ and $X$ is a globally
	$F$-regular variety over~$k$, then so is~$Y$.
\end{thm-intro}

The proof of the linear reductivity in \cref{intro:tame}, given in \cref{sec:tame}, goes as follows. By \cref{lem:trace-generates}, the trace of a quasi-torsor
cover  freely generates the relative
dualizing sheaf\,; over a field, this is due to Carvajal-Rojas \cite[\S3]{carvajal_rojas_finite_torsors_over_strongly_f_regular_singularities}.
 Hence a splitting of $\cO_X\to\pi_*\cO_Y$, which exists since $X$ is a splinter, is the
trace multiplied by a global function on~$Y$, which by assumption comes from~$X$. Thus the trace of~$1$ is invertible. By Maschke's theorem in the form of \cref{lem:maschke-base}, this condition characterizes linear reductivity.
We also give in \cref{sec:second-proof} another proof, in the special case 
$R=k$, that goes through a d\'evissage argument starting with excluding the
existence of $\alpha_p$- and $\ZZ/p$-quasi-torsor covers as in the theorem.

In particular,  for a connected proper splinter~$X$ over~$k$, a big open
$U\subseteq X$ and $x\in U(k)$, every finite quotient of $\piN(U,x)$ is linearly reductive, as the
proof of \cref{thm:piN2} shows. It follows that the regular locus of a connected proper splinter over $\bar k$ carries 
no nontrivial torsor under
$\alpha_p$ or under any infinitesimal unipotent group scheme, and, by taking Galois closures, 
no connected \'etale cover of degree divisible by~$p$.
The unipotent and wild phenomena responsible for much of the pathology of characteristic-$p$ geometry
are thus absent from the covers of regular loci of proper splinters over~$\bar k$. 
Furthermore, 
 together with Nagata's theorem
\cite[IV, \S3, Thm.~3.6]{demazure_gabriel_groupes_algebriques}
 and \cref{lem:crepancy}, \cref{intro:tame} shows that
every quasi-torsor cover with integral total space of a connected proper splinter over $\bar k$ refines into quasi-\'etale steps of degree
prime to~$p$ and $\mu_p$-quasi-torsor covers.
\medskip

Our next result is the analogue, for quasi-torsors under arbitrary finite group schemes, of the
quasi-\'etale tower theorems of
\cite{greb_kebekus_peternell_etale_fundamental_groups_of_klt_spaces, bhatt_carvajal-rojas_graf_schwede_tucker}.

\begin{thm-intro}[Towers of quasi-torsor covers]
	\label{thm:main}
	Let	$R$ be a Noetherian complete local domain whose residue field is algebraically
	closed of characteristic $p>0$, and let 
	$X$ be an integral normal
	scheme, projective over $R=H^0(X,\cO_X)$, whose generic fiber has dimension~$e\ge 1$.
	Assume that one of the following holds\,:
	\begin{enumerate}
		\itemcase{p}{thm:main-p} $R$ has characteristic~$p$, and $X$ is globally $F$-regular\,;
		\itemcase{0}{thm:main-0} $R$ has characteristic~$0$, all fibers of $X\to\Spec R$ have dimension~$e$, and there is a $\bbQ$-divisor $\Delta\ge0$ on $X$
		such that $-(K_X+\Delta)$ is ample and such that for every Cartier divisor $D>0$ on $X$
		there is an $\varepsilon>0$ for which $(X,\Delta+\varepsilon D)$ is globally $+$-regular.
	\end{enumerate}
	Consider a sequence of finite surjective morphisms of $S_2$ schemes whose irreducible components all dominate~$X$
	\[
	X=Y_0\xleftarrow{\ f_0\ } Y_1\xleftarrow{\ f_1\ } Y_2\xleftarrow{\ f_2\ }\cdots,
	\]
	where each $f_i$ is a quasi-torsor under a finite locally free $R$-group scheme $G_i$ and
	where $R=H^0(Y_i,\cO_{Y_i})$ for all~$i$. Then each $Y_i$ is an integral splinter, each $G_i$ is linearly reductive over~$R$, and the sequence stabilizes, i.e.\ $f_i$ is an isomorphism for all sufficiently large~$i$. 
\end{thm-intro}

We refer to \cref{thm:unif}, where in addition we provide a bound on $\deg(Y_n\to X)$ uniform in~$n$. Such a bound does not depend on the tower in case~\ref{thm:main-p}, and is expected not to in case~\ref{thm:main-0}, see~\S\ref{sec:unif-statement}.
In \cref{thm:main-k}, we also provide a version of \cref{thm:main} where $X$ is assumed to be a globally $F$-regular projective variety over the (not necessarily algebraically closed) field $k$.
If $R=\bar k$ and the $Y_i$ are integral, the hypothesis $H^0(X,\cO_X)=H^0(Y_i,\cO_{Y_i})$ is automatic, so that \cref{thm:main} asserts that every tower of quasi-torsor covers of~$X$ with integral members stabilizes.

Case~\ref{thm:main-0} is the setting of the result of Cai--Lee--Ma--Schwede--Tucker 
\cite[Thm.~7.4]{cai-lee-ma-schwede-tucker}, whose proof uses a perfectoid variant of the $F$-signature.
The new feature of \cref{thm:main} in case~\ref{thm:main-0} is that the steps need not be
quasi-\'etale, e.g.\ if $G_i=\mu_{p,R}$, which is ramified over~$p$ (see \cref{rmk:unif-clmst}).
\Cref{thm:main} has no local counterpart, even if stabilization is weakened, as in
\cite{greb_kebekus_peternell_etale_fundamental_groups_of_klt_spaces,bhatt_carvajal-rojas_graf_schwede_tucker},
to the requirement that all but finitely many steps be torsors\,: \cref{ex:local-tower} exhibits an
infinite tower of $\mu_p$-quasi-torsor covers of $\Spec\bar k[[u^p,v^p]]$ with strongly $F$-regular
members, infinitely many of whose steps are not torsors.
The key point in the global case is the grading of section rings\,: the trace of the induced
quasi-torsor cover of cones is homogeneous of degree~$0$, hence maps the maximal ideal at the vertex
into the maximal ideal, see \cref{lem:cone-cover-unif}(ii). Together with the surjectivity of the
trace, which here follows from the linear reductivity of the structure groups
(\cref{intro:tame} and \cref{lem:cone-cover-unif}(i)), this allows the $F$-signature, or its
perfectoid variant, to be multiplied by the degree at each step
\cite[Thm.~4.8]{carvajal_rojas_finite_torsors_over_strongly_f_regular_singularities},
\cite[Prop.~6.12]{cai-lee-ma-schwede-tucker}.

Finally, we pinpoint that global $F$-regularity is used in the proofs of \cref{thm:piN}  and \cref{thm:main}\ref{thm:main-p} only to ensure that
the local ring at the vertex of a cone over~$X$ is strongly $F$-regular,
 hence has
positive $F$-signature \cite[Thm.~0.2]{aberbach_leuschke_signature}. 
All other steps only use
that $X$ is a splinter.

\begin{conventions}
	Throughout, $k$ is a field of characteristic $p>0$ with algebraic closure $\bar k$, and a
	\emph{variety} is an integral separated scheme of finite type over $k$. 
	For an affine group scheme $\pi_G\colon G\to S$, we write $\cO(G)\coloneqq\pi_{G*}\cO_G$.
	For a generically finite morphism $\pi \colon Y \to X$ with $X$ integral, we write $\deg(\pi)$ for its generic degree.
	For a point $s$ of a scheme~$S$, we write $\fm_s$ for the maximal ideal of~$\cO_{S,s}$, and for a local ring~$A$, we write $\fm_A$ for its maximal ideal.
\end{conventions}

\begin{ack}
	Thanks to Manuel Hoff and Johannes Krah for discussions in
	early 2025 when this project began.
	I am very grateful to All Souls College, Oxford, for providing a stimulating and inspiring environment during a Visiting Fellowship in autumn 2025.
\end{ack}

\begin{AI}
My original approach to \cref{intro:tame}, which only considered the positive characteristic setting $R=k$,  is the one given
in \cref{sec:second-proof}. 
This was used to establish \cref{thm:main-k} and \cref{thm:piN}. 
In September 2026, I asked Claude how an equivariant trace map, which I had previously constructed in joint work with Johannes Krah and hoped to use for this project, could come into play 
and Claude suggested an alternative proof of \cref{intro:tame} based on Maschke's theorem. 
This immediately led me to the characteristic-free proof of \cref{intro:tame} presented here and paved the way for establishing \cref{thm:main}.
The same suggestion initiated forthcoming work~\cite{vial_tame} on the local situation. There, I use a version of Maschke's theorem to show that any quasi-torsor cover of $\Spec R^\star$ is a torsor, where $R\subseteq R^\star$ is the maximal cover of a strongly $F$-regular germ~$R$ constructed by Carvajal-Rojas \cite[Main~Thm.]{carvajal_rojas_finite_torsors_over_strongly_f_regular_singularities}.
LLMs were also used for searching the literature,  for the
initial drafting of certain passages, for copy-editing and checking for correctness.
I take full responsibility for the content and correctness of the paper.
\end{AI}

\section{Preliminaries} \label{sec:preli}

\subsection{Reflexive sheaves} \label{sec:reflexive_sheaves}
A coherent sheaf $\cF$ on an integral locally Noetherian scheme is said to be \emph{reflexive} if the natural map
$\cF\to\cF^{\vee\vee}$ to its double dual is an isomorphism. 
For an
integral normal locally Noetherian scheme $W$, every sheaf of the form $\cHom_{\cO_W}(\cF,\cO_W)$ with $\cF$ coherent is reflexive, and a
coherent sheaf on $W$ is reflexive if and only if it is torsion-free and~$S_2$.
 If 
$j\colon V\hookrightarrow W$ is an open immersion of a big open subset,
a reflexive sheaf $\mathcal G$ on $W$ satisfies
$\mathcal G\xrightarrow{\ \sim\ }j_*j^*\mathcal G$ and conversely, for
$\mathcal H$ reflexive on $V$, the pushforward $j_*\mathcal H$ is
coherent reflexive, so that $j_*$ and~$j^*$ are mutually inverse equivalences
between reflexive sheaves on $V$ and on $W$. 
We refer to \cite[\href{https://stacks.math.columbia.edu/tag/0AVT}{Tag~0AVT}]{stacks-project}.

\subsection{Relative dualizing sheaf for finite morphisms} \label{sec:upper_shriek}

Let $\pi\colon Y \to X$ be a finite morphism of schemes such that $\pi_*\cO_Y$
is finitely presented as an $\cO_X$-module.
This holds if $X$ is locally
Noetherian or  if $\pi$ is finite locally free.
Since $\pi$ is affine, $\pi_*$ identifies the category $\Qcoh(Y)$ of quasi-coherent sheaves on $Y$
with the category of quasi-coherent $\pi_*\cO_Y$-modules on~$X$.
 The puschforward functor $\pi_*\colon \Qcoh(Y) \to \Qcoh(X)$ then admits a right adjoint functor that is denoted
\[
\pi^{!}\colon\Qcoh(X)\to\Qcoh(Y),\qquad
\pi^{!}\cF\coloneqq\cHom_{\cO_X}(\pi_*\cO_Y,\cF).
\]
By definition, the \emph{relative dualizing sheaf} of
$\pi$ is the $\cO_Y$-module
\[
\omega_{\pi}=\omega_{Y/X}\coloneqq\pi^{!}\cO_X=\cHom_{\cO_X}(\pi_*\cO_Y,\cO_X).
\]
Its formation commutes with flat base change on~$X$. 
(If $X$ is locally Noetherian, then $\omega_{\pi}$ is moreover coherent and $\omega_{\pi}$ is $\mathcal{H}^{0}$ of Grothendieck's
exceptional inverse image of $\cO_X$ along $\pi$ 
\cite[{}III.6]{hartshorne_residues_and_duality}.)
Moreover, the formation of $\omega_\pi$ is functorial\,: if $\pi'\colon Z\to Y$ is a further finite
morphism with $\pi'_*\cO_Z$ finitely presented, then right adjoints of the
composed pushforwards compose, giving a canonical isomorphism
$(\pi\circ\pi')^{!}\isom\pi'^{!}\circ\pi^{!}$.

If $\pi\colon Y \to X$ is finite locally free, then $\pi_*\cO_Y$ is
locally free of finite rank. In this case, the formation of
$\omega_{\pi}$ commutes with  arbitrary base change on~$X$ 
(and not merely with flat base change), 
and  the natural map
$\pi^{*}\cF\otimes_{\cO_Y}\omega_{\pi}\to\pi^{!}\cF$ is an isomorphism for every quasi-coherent $\cO_X$-module $\cF$.

\subsection{Globally $F$-regular schemes, globally $+$-regular schemes and splinters}\label{sec:splinters}

\subsubsection{Globally $F$-regular schemes}
A scheme~$X$ over $k$ is said to be \emph{$F$-finite} if the absolute
Frobenius $F\colon X\to X$ is finite, and \emph{$F$-split} if $\cO_X\to F_\ast\cO_X$
admits an $\cO_X$-linear splitting.
A normal scheme~$X$ over $k$ is \emph{globally $F$-regular} 
if it is $F$-finite and if for every effective
divisor $D$ on~$X$ there is $e>0$ such that the composition
$\cO_X\to F^e_*\cO_X\hookrightarrow F^e_*\cO_X(D)$ splits.
More generally, for an effective
$\bbQ$-divisor~$\Delta$ on an $F$-finite normal scheme $X$ over $k$, the pair $(X,\Delta)$ is \emph{globally $F$-regular} if for
every effective divisor~$D$ on~$X$ there is $e>0$ such that
$\cO_X\to F^e_*\cO_X(\lceil(p^e-1)\Delta\rceil+D)$ splits~\cite[Def.~3.1(i)]{schwede_smith_globally_f_regular_and_log_Fano_varieties}. The notion of global $F$-regularity was introduced by Smith
\cite{smith_globally_f_regular_varieties_applications_to_vanishing_theorems_for_quotients_of_fano_varieties}\,;
for instance, projective toric
varieties over $F$-finite fields are globally $F$-regular~\cite[Prop.~6.4]{smith_globally_f_regular_varieties_applications_to_vanishing_theorems_for_quotients_of_fano_varieties}.
Schwede--Smith showed that globally $F$-regular projective varieties are log Fano, while complex
log Fano varieties are of globally $F$-regular type, i.e.\ their reductions modulo~$p\gg0$ are
globally $F$-regular~\cite[Thms.~1.1 \& 1.2]{schwede_smith_globally_f_regular_and_log_Fano_varieties}. 

\subsubsection{Splinters}
A Noetherian scheme~$X$ is called a \emph{splinter} if for all finite surjective morphisms $\pi \colon Y \to X$, the canonical map $\cO_X \to \pi_*\cO_Y$ splits as $\cO_X$-modules. 
Obviously, both $F$-finite splinters and globally
$F$-regular varieties over $k$ are $F$-split.  
A Noetherian globally $F$-regular scheme over $k$ is a splinter
\cite[Lem.~6.14]{bhatt_et_al_globally_+_regular_varieties_and_mmp_for_threefolds_in_mixed_char}, and its
local rings are strongly $F$-regular. Conversely, it is an open question whether $F$-finite splinters over $k$ are globally $F$-regular
\cite[Rmk.~6.16]{bhatt_et_al_globally_+_regular_varieties_and_mmp_for_threefolds_in_mixed_char}.

By Zariski's Main Theorem, the splinter property passes to dense open subschemes\,: if $X$ is a splinter
and $U\subseteq X$ is a dense open, then $U$ is a splinter\,; see e.g.\ \cite[Lem.~4.1(i)]{krah-vial}.
Splinters have the following local properties. Any splinter is normal, and splinters of finite type over a field of positive
characteristic are moreover Cohen--Macaulay and $F$-rational\,; see e.g.\ \cite[Prop.~3.5]{krah-vial} and the references therein.
Globally, the structure sheaf of a proper splinter $X$ over a Noetherian ring has vanishing positive-degree cohomology, provided the closed points of~$X$ have residue fields of positive characteristic\,:
the following lemma removes the excellence and dualizing complex assumptions of
\cite[Prop.~6.13]{bhatt_et_al_globally_+_regular_varieties_and_mmp_for_threefolds_in_mixed_char}.

\begin{lemma}\label{lem:splinter-vanishing}
	Let $R$ be a Noetherian ring, and let $X$ be a proper $R$-scheme whose closed points have residue fields of positive characteristic. If $X$ is a splinter, then $H^i(X,\cO_X)=0$ for all $i>0$.
\end{lemma}
\begin{proof}
	Since $H^i(X,\cO_X)$ is a finite $R$-module
	\cite[\href{https://stacks.math.columbia.edu/tag/02O6}{Tag~02O6}]{stacks-project} whose formation commutes with flat base change
	\cite[\href{https://stacks.math.columbia.edu/tag/02KH}{Tag~02KH}]{stacks-project}
	 and since $X_{R_\fm}$ is a splinter for every maximal ideal~$\fm$ of~$R$
	\cite[Lem.~7.3]{krah-vial}, we may assume that $R$ is local with residue field of characteristic~$p>0$, and that $X$ is nonempty. 
	As $X$ is normal,
	working on each connected component, we may assume that $X$ is integral, so that $p$ is
	either zero or a nonzerodivisor on~$\cO_X$.
	
	If $p=0$ on~$X$, then $X \to \Spec R/p$ is a proper morphism of Noetherian $\mathbb{F}_p$-schemes and
	\cite[Thm.~1.5]{bhatt_derived_splinters_in_positive_characteristic} provides a finite
	surjective $\pi\colon Y\to X$ such that $\pi^*\colon H^i(X,\cO_X)\to H^i(Y,\cO_Y)$ is zero for
	all $i>0$. Since $\cO_X\to\pi_*\cO_Y$ splits, this map is also injective, and we obtain $H^i(X,\cO_X)=0$.
	
	Assume now that $p$ is a nonzerodivisor on~$\cO_X$, and set $X_p\coloneqq X\times_\ZZ\mathbb{F}_p$.
	Since $H^i(X,\cO_X)\xrightarrow{\,p\,}H^i(X,\cO_X)\to H^i(X_p,\cO_{X_p})$ is exact and
	$H^i(X,\cO_X)$ is a finite module over the local ring~$R$, whose maximal ideal contains~$p$, it
	suffices by Nakayama's lemma to show that $H^i(X_p,\cO_{X_p})=0$ for $i>0$. Write $R$, with
	maximal ideal~$\fm$, as the filtered union of its local subrings
	$R_\lambda\coloneqq(A_\lambda)_{A_\lambda\cap\fm}$, where $A_\lambda$ runs over the finitely
	generated subrings of~$R$. By
	\cite[\href{https://stacks.math.columbia.edu/tag/01ZM}{Tag~01ZM} \&
	\href{https://stacks.math.columbia.edu/tag/081F}{Tag~081F}]{stacks-project}, $X$ is the base
	change of a proper $R_{\lambda_0}$-scheme $X'$\,; for $\lambda\ge\lambda_0$, the scheme-theoretic image
	$X_\lambda$ of $X$ in $X'\times_{R_{\lambda_0}}R_\lambda$ is, like~$X$, integral and $p$-torsion-free,
	and $X=\lim_\lambda X_\lambda$. So a class $\xi\in H^i(X_p,\cO_{X_p})$ with $i>0$ is the image of
	some $\xi_\lambda\in H^i(X_{\lambda,p},\cO_{X_{\lambda,p}})$
	\cite[\href{https://stacks.math.columbia.edu/tag/0CUF}{Tag~0CUF} \&
	\href{https://stacks.math.columbia.edu/tag/0A37}{Tag~0A37}]{stacks-project}. Moreover, $X_\lambda$
	is proper and surjective over $\Spec T_\lambda$, where the image $T_\lambda$ of $R_\lambda$ in
	$H^0(X_\lambda,\cO_{X_\lambda})$ is a local domain of mixed characteristic $(0,p)$, essentially of
	finite type over~$\ZZ$, hence excellent
	\cite[\href{https://stacks.math.columbia.edu/tag/07QW}{Tag~07QW} \&
	\href{https://stacks.math.columbia.edu/tag/07QU}{Tag~07QU}]{stacks-project} with a dualizing
	complex \cite[\href{https://stacks.math.columbia.edu/tag/0BFR}{Tag~0BFR}]{stacks-project}. So
	\cite[Prop.~3.1(a)]{bhatt_et_al_globally_+_regular_varieties_and_mmp_for_threefolds_in_mixed_char}
	provides a finite surjective $Y_\lambda\to X_\lambda$ killing~$\xi_\lambda$. Its base change
	$\pi\colon Y\to X$ is finite and surjective and kills~$\xi$. Since $\cO_X\to\pi_*\cO_Y$ splits, so does $\cO_{X_p}\to\pi_*\cO_{Y_p}$, and we obtain $\xi=0$.
\end{proof}

\subsubsection{Globally $+$-regular schemes}
A pair $(X,\Delta)$ consisting of a normal excellent Noetherian scheme~$X$ with closed points having positive-characteristic residue fields and an
effective $\bbQ$-divisor~$\Delta$ on~$X$ is \emph{globally $+$-regular} if for every finite
surjective morphism $f\colon Y\to X$ with $Y$ normal and every irreducible component of~$Y$
dominating an irreducible component of~$X$, the map
$\cO_X\to f_*\cO_Y(\lfloor f^*\Delta\rfloor)$ splits
\cite[Def.~6.1]{bhatt_et_al_globally_+_regular_varieties_and_mmp_for_threefolds_in_mixed_char}. In
particular, $X$ is a splinter if and only if $(X,0)$ is globally $+$-regular since every finite
surjective morphism $Y\to X$ is dominated by the disjoint union of the
normalizations of the irreducible components of~$Y$ dominating an irreducible component of~$X$.
This union is finite over~$Y$ since $X$ is excellent (cf.\ \cite[Rmk.~3.3]{krah-vial}).

\subsection{Ascending global $F$-regularity and the splinter property}\label{sec:ascent}
The following ascent lemma is the key input from \cite{krah-vial} that we will need.

\begin{lemma}[{\cite[Lem.~5.3(b) \& Prop.~6.4(b)]{krah-vial}}] \label{lem:lift}
	Let $\pi \colon Y \to X$ be a finite surjective morphism of Noetherian schemes. 
	Assume that $\omega_{Y/X} \isom \cO_Y$ in $\Coh(Y)$, and either $H^0(Y,\cO_Y)$ is a field or $H^0(X,\cO_X) = H^0(Y,\cO_Y)$.
	\begin{enumerate}
		\item If $X$ is a splinter, then $Y$ is a splinter.
		\item If $X$ is a globally $F$-regular variety over $k$, then $Y$ is a globally $F$-regular variety.
	\end{enumerate}
\end{lemma}

This lemma is, in combination with three further inputs, the key to the proof in
\cite[Thm.~A]{krah-vial} that the Nori fundamental group scheme of a connected proper splinter
over $\bar k$ is trivial. The first \cite[Lem.~7.1]{krah-vial} is the triviality of the relative
dualizing sheaf of a torsor under a finite locally free group scheme over a ring with trivial
Picard group\,; \cref{lem:trace-generates} below recovers it and extends it to quasi-torsor covers.
The second and the third, which play no role in the proofs of \cref{thm:piN,intro:tame,thm:main}, are of global nature\,: the multiplicativity of the Euler
characteristic along finite torsors over proper varieties \cite[Lem.~8.1]{krah-vial}, and the
vanishing of the positive-degree cohomology of the structure sheaf of a proper splinter in positive characteristic
(\cref{lem:splinter-vanishing}).
We refer to \cref{thm:KV-mixed} below for a version of \cite[Thm.~A]{krah-vial} valid in mixed characteristic.

\section{Quasi-torsors and quasi-torsor covers} \label{sec:q-torsors}

\subsection{Setup}
Here and throughout, a \emph{torsor} under a finite locally free group scheme
$G$ means an fppf torsor under $G$ with $G$ acting on the right.
We spell out what we mean by \emph{quasi-torsor} and \emph{quasi-torsor cover}.

\begin{definition}\label{def:quasitorsor}
	Let $S$ be a scheme and	let $G$ be a finite locally free $S$-group scheme.
	\begin{enumerate}
		\item 	A morphism $\pi\colon Y\to X$ of locally Noetherian
		$S$-schemes, together with a $G$-action on $\pi^{-1}(U)$ making $\pi^{-1}(U)\to U$ a
		$G$-torsor for some big open subset $U\subseteq X$, is a \emph{quasi-torsor} under~$G$.
		\item A morphism $\pi\colon Y\to X$ of locally Noetherian schemes is a \emph{cover} if
		$X$ is integral normal and $\pi$ is a finite surjective morphism such that $\pi_*\cO_Y$ is reflexive, equivalently (\cref{rmk:cover-reflexive}) such that $Y$ is an $S_2$ scheme whose irreducible components all dominate~$X$.
		\item A morphism $\pi\colon Y\to X$ of locally Noetherian
		$S$-schemes is a \emph{quasi-torsor cover} under $G$ if it is a quasi-torsor under $G$ and a
		cover. (Note that if $X$ is integral normal, then a $G$-torsor over $X$ is a quasi-torsor cover.)
	\end{enumerate}
\end{definition}

\begin{remark}\label{rmk:cover-reflexive}
	 Let $\pi\colon Y\to X$ be a finite morphism of locally Noetherian schemes with $X$ integral normal.
	Recall from \S\ref{sec:reflexive_sheaves} that $\pi_{\ast}\cO_Y$ is reflexive if and only if it is
torsion-free and satisfies Serre's condition $S_2$ as a coherent $\cO_X$-module.
Torsion-freeness means that every associated point of $Y$ lies over the generic
point of $X$\,; it implies that every irreducible component of~$Y$ dominates~$X$, and it is
an equivalence provided $Y$ is $S_1$, e.g.\ $S_2$.
Given torsion-freeness, the $S_2$ condition on $\pi_{\ast}\cO_Y$ is
equivalent to $Y$ being an $S_2$ scheme in the sense of
\cite[\href{https://stacks.math.columbia.edu/tag/033P}{Tag~033P}]{stacks-project}.
Indeed, the depth of $\pi_{\ast}\cO_Y$ as an $\cO_X$-module at a point $x \in X$ equals the minimum of the depths of its localizations at the points of~$Y$ above~$x$
\cite[\href{https://stacks.math.columbia.edu/tag/0AUK}{Tag~0AUK}]{stacks-project},
and $\dim\cO_{Y,y}=\dim\cO_{X,\pi(y)}$ by going-down
over the normal base~$X$, since every irreducible component of~$Y$ dominates~$X$.
This shows that the two Serre conditions read off the same inequalities, and hence that the two descriptions of a
cover in \cref{def:quasitorsor}(ii) agree. 
\end{remark}

A morphism $\pi\colon Y\to X$ of locally Noetherian
$S$-schemes is \emph{quasi-\'etale} if there is a big open subset
$U\subseteq X$ such that
$\pi^{-1}(U)\to U$ is \'etale.   

 \begin{remark}\label{rmk:nonnormal-torsor} While the total space $Y$ of a
 	quasi-\'etale cover is necessarily normal,
 	the total space of a torsor cover under
 	an infinitesimal group scheme over a regular base need not be normal.
 	Indeed, assume $p$ odd and let $U\subseteq\mathbb A^1=\Spec k[x]$ be the complement
 	of the vanishing locus of $1+x^2$.  Let
 	$V=\Spec_U\,\cO_U[t]/(t^p-1-x^2)$
 	be the $\mu_p$-torsor over $U$ associated with the unit $1+x^2$ by Kummer
 	theory. Setting $s=t-1$ gives $s^p=x^2$, which is irreducible since $p$ is odd. So $V$ is an integral hypersurface
 	over the smooth curve $U$.  In particular, $V$ is 
 	 Cohen--Macaulay, hence~$S_2$. On the other hand,
 	$V$ has a cusp at $(x,s)=(0,0)$ and so is not normal.
 	Thus $V\to U$ is a torsor cover with nonnormal total space.
 \end{remark}

\subsection{Reflexive extension of torsors}
The following lemma shows that $G$-coactions and algebra structures on reflexive sheaves extend uniquely across big opens.

\begin{lemma}\label{lem:reflexive-coaction}
	Let $S$ be a scheme, let $G$ be a finite locally free $S$-group scheme and let
	$j\colon V\hookrightarrow W$ be a big open of an integral normal locally Noetherian $S$-scheme~$W$.
	Then $j_*$ and~$j^*$ are mutually inverse equivalences between reflexive sheaves, resp.\ algebras, with $G$-coaction on~$V$ and on~$W$. Moreover, $j_*$ and $j^*$ commute with taking invariants.
\end{lemma}
\begin{proof}
	Let $h\colon W\to S$ be the structure morphism and write
	$G\times_SW=\Spec_W\mathcal B$ with  $\mathcal B\coloneqq h^*\cO(G)$ \cite[\href{https://stacks.math.columbia.edu/tag/01LX}{Tag~01LX}]{stacks-project}.
	Since $\mathcal B$ is finite locally free, tensoring with~$\mathcal B$ preserves reflexive sheaves on~$W$. Let $\mathcal M$ be quasi-coherent and $\mathcal N$ be reflexive on~$W$.
	By adjunction and since $\mathcal N\isom j_*j^*\mathcal N$, restriction to $V$ gives a bijection $\Hom_W(\mathcal M,\mathcal N)\isom\Hom_V(j^*\mathcal M,j^*\mathcal N)$.
	
	Now, that $j_*$ and $j^*$ are mutually inverse equivalences follows from the fact that, for a reflexive sheaf, resp.\ a reflexive algebra, $\mathcal A$ on~$W$, the coaction
	$\mathcal A\to\mathcal A\otimes\mathcal B$, resp.\ the multiplication and the unit, their
	axioms, and the conditions on morphisms only involve maps with targets of the form $\mathcal A$,
	$\mathcal A\otimes\mathcal B$ or $\mathcal A\otimes\mathcal B\otimes\mathcal B$, which are reflexive. 
	
	Finally, let $\mathcal A$ be a reflexive sheaf on~$W$ with $G$-coaction. Then $\mathcal A^G$ is the kernel of the difference
	$\mathcal A\to\mathcal A\otimes\mathcal B$ of the coaction and $\id\otimes1$. Since $j_*j^*$ is
	left exact and $\mathcal F\isom j_*j^*\mathcal F$ for $\mathcal F=\mathcal A$ and
	$\mathcal F=\mathcal A\otimes\mathcal B$, we get $\mathcal A^G=j_*\bigl((j^*\mathcal A)^G\bigr)$.
\end{proof}

\begin{remark}\label{rmk:action-extends}
	The $G$-action on a quasi-torsor cover $\pi\colon Y\to X$, which by definition is given over a big
	open $j\colon U\hookrightarrow X$ where $\pi$ is a torsor, thus extends uniquely to an action on~$Y$
	over~$X$. Indeed, $\pi_*\cO_Y$ is reflexive by definition, so the coaction on $j^*\pi_*\cO_Y$ extends
	uniquely to $\pi_*\cO_Y=j_*j^*\pi_*\cO_Y$ by \cref{lem:reflexive-coaction}. Since finite morphisms
	with a $G$-action correspond to coherent algebras with a $G$-coaction, quasi-torsor covers of~$X$
	under~$G$ correspond to coherent reflexive $\cO_X$-algebras with $G$-coaction whose restriction to
	some big open of~$X$ is the algebra of a $G$-torsor. 
\end{remark}
\begin{remark}\label{rmk:invariants}
	For a $G$-torsor $\pi\colon Y\to X$, the map $\cO_X\to(\pi_*\cO_Y)^G$ is an isomorphism by fppf
	descent. This also holds for a quasi-torsor cover~$\pi\colon Y\to X$. Indeed, with $j\colon U\hookrightarrow X$ a big open over which $\pi$ is a torsor,
	$(j^*\pi_*\cO_Y)^G=\cO_U$ by the torsor case, and $j_*\cO_U=\cO_X$ since $X$ is normal. Using that $j_*$ commutes
	with taking invariants by \cref{lem:reflexive-coaction}, we get
	$(\pi_*\cO_Y)^G=j_*\bigl((j^*\pi_*\cO_Y)^G\bigr)=j_*\cO_U=\cO_X$.
\end{remark}

The following correspondence is well known to experts (for $X$ local over an algebraically closed
field, see \cite[Term.~1.3 \& Lem.~4.1]{carvajal_rojas_finite_torsors_over_strongly_f_regular_singularities})
 and provides a dictionary between quasi-torsor covers and the torsors they restrict~to over the regular locus. Its one nonformal ingredient is the purity theorem of \cite[Lem.~2]{moret-bailly_purete}, extending Zariski--Nagata purity \cite[Exp.~X, \S3]{SGA2} from the \'etale case to finite torsors under a finite locally free group scheme\,; see \cite{marrama} for a proof, and also \cite[Ch.~II, Prop.~7]{nori_the_fundamental_group_scheme}.

\begin{proposition}
	\label{prop:quasitorsor-correspondence}
	Let $S$ be a scheme, let $G$ be a finite locally free $S$-group scheme and let $X$ be an integral
	normal locally Noetherian $S$-scheme whose regular locus
	$j\colon X^{\reg}\hookrightarrow X$ is open.
	Then the restriction and reflexive extension functors
	\[
	\mathrm{res}\colon (\pi_Y\colon Y\to X)\longmapsto Y|_{X^{\reg}}
	\qquad\text{and}\qquad
	\mathrm{ext}\colon (\pi_V\colon V \to X^{\reg}) \longmapsto \Spec_X\!\bigl(j_\ast\pi_{V\ast}\cO_V\bigr)
	\]
	give mutually quasi-inverse equivalences between categories of $G$-equivariant objects
	\[
	\{\,\text{quasi-torsor covers of }X\text{ under }G\,\}
	\;\simeq\;
	\{\,\text{torsors under }G\text{ over }X^{\reg}\,\}.
	\]
	More generally, for every big open $j_U \colon U\hookrightarrow X$ and every torsor
	$\pi_V\colon V\to U$ under~$G$, the morphism $\Spec_X(j_{U\ast}\pi_{V\ast}\cO_V)\to X$ is a quasi-torsor
	cover under~$G$ whose restriction to~$U$ is~$V$.
	Moreover, under this equivalence, $Y$ is integral if and only if $V$ is integral.
\end{proposition}

For $G$ infinitesimal, the total space of a quasi-torsor is
typically not the normalization of $X$ in a field extension.
For instance, the
trivial $\mu_p$-torsor has nonreduced total space. 
In fact, even when the total space is integral it need not be normal\,; see \cref{rmk:nonnormal-torsor}. This explains the use of the reflexive
pushforward~$j_\ast\pi_{V\ast}\cO_V$.

\begin{proof}[Proof of \cref{prop:quasitorsor-correspondence}]
	Since $X$ is normal, $X^{\reg}$ is a big open in $X$.
	By \cref{lem:reflexive-coaction}, $j_\ast$ and $j^\ast$ are mutually inverse equivalences between reflexive algebras with $G$-coaction on~$X^{\reg}$ and on~$X$. 
	By \cref{rmk:action-extends}, in order to establish the equivalence, it suffices to show that 
	the restriction to $X^{\reg}$ of a quasi-torsor cover $\pi_Y\colon Y\to X$ under $G$ is a torsor under $G$.
Let
$U_0\subseteq X$ be a big open over which $\pi_Y$ restricts to a torsor. Then $Y|_{X^{\reg}}$ restricts to a torsor over the big open
$U_0\cap X^{\reg}$ of the regular scheme~$X^{\reg}$. By purity
\cite[Lem.~2]{moret-bailly_purete} this torsor extends to a torsor~$\widetilde V\to X^{\reg}$.  Now $(\pi_{Y|_{X^{\reg}}})_\ast\cO_{Y|_{X^{\reg}}}=j^\ast(\pi_{Y\ast}\cO_Y)$ and the locally free
$\pi_{\widetilde V\ast}\cO_{\widetilde V}$ are both reflexive algebras with $G$-coaction and they agree over
the big open $U_0\cap X^{\reg}$ of~$X^{\reg}$. By \cref{lem:reflexive-coaction}, they are therefore isomorphic as algebras with $G$-coaction.
Thus $\widetilde V=Y|_{X^{\reg}}$, so
$\mathrm{res}(Y)=Y|_{X^{\reg}}$ is a torsor over~$X^{\reg}$.

	Let now 	$\pi_V\colon V\to U$ be a torsor under~$G$ over  a big open $j_U \colon U\hookrightarrow X$. Then $j_{U\ast}\pi_{V\ast}\cO_V$ is a reflexive algebra with $G$-coaction by \cref{lem:reflexive-coaction}.
	Its restriction to~$U$ is $\pi_{V\ast}\cO_V$, and we conclude with \cref{rmk:action-extends} that $\Spec_X(j_{U\ast}\pi_{V\ast}\cO_V)\to X$ is a quasi-torsor
	cover under~$G$ whose restriction to~$U$ is~$V$.

   Finally, note that the torsor $V=Y|_{X^{\reg}}$ surjects onto the nonempty~$X^{\reg}$.
	Thus $V$ is a nonempty open subscheme of $Y$, hence integral whenever $Y$ is. Conversely, let $V$ be integral. Since $V$ is finite and dominant over
	$X^{\reg}$, it has generic fiber $\Spec k(V)$, so the stalk of $\pi_{V\ast}\cO_V$ at the generic
	point of $X$ is the field $k(V)$. The torsion-free algebra $j_\ast\pi_{V\ast}\cO_V$ therefore
	embeds into the constant sheaf $k(V)$ and is a sheaf of domains, so $Y$ is integral.
\end{proof}

\subsection{Intermediate quotients of quasi-torsor covers}
	\begin{lemma}\label{lem:crepancy}
	Let $S$ be a scheme and let $\pi\colon Y\to X$ be a
	quasi-torsor cover under a finite locally free $S$-group scheme~$G$,
	and let $H\subseteq G$ be a closed subgroup scheme, finite locally free
	over $S$.
	Set $Z\coloneqq Y/H=\Spec_X\bigl((\pi_*\cO_Y)^{H}\bigr)$, with natural factorization
	\[
	\pi\colon Y\xrightarrow{\,r\,}Z\xrightarrow{\,q\,}X.
	\]
	Then $r$ is a finite surjective quasi-torsor under~$H$\,; more precisely, if $\pi$ is a $G$-torsor over a big
	open $U\subseteq X$, then $r$ restricts to an $H$-torsor $\pi^{-1}(U)\to q^{-1}(U)$. Moreover, $q$ is a cover which is a
	quasi-torsor cover under $G/H$ if $H$ is normal in $G$. If $Y$ is integral, resp.\ normal, then so is~$Z$\,; in particular, if $Y$ is integral and normal, then $r$ is a quasi-torsor cover.
\end{lemma}

\begin{proof}
	Since $\pi$ is a quasi-torsor cover, $\mathcal A\coloneqq\pi_*\cO_Y$ is a reflexive $\cO_X$-algebra,
	and $G$ acts on~$Y$ over $X$ by \cref{rmk:action-extends}. 
	Let $\sigma$ denote the coaction of $H$ on~$\mathcal A$, so that
	$\mathcal A^{H}=\ker\bigl(\sigma-(\id\otimes1)\bigr)$.

The $\cO_X$-module
	$\mathcal A^{H}$ is a submodule of the finite torsion-free module $\mathcal A$ over the locally
	Noetherian~$X$. So $\mathcal A^{H}$ is finite and torsion-free, and $Z=\Spec_X(\mathcal A^{H})$ is finite over~$X$. The morphism~$r$ is
	finite since $\mathcal A$ is finite over~$\mathcal A^{H}$, and it is surjective by lying-over
	\cite[\href{https://stacks.math.columbia.edu/tag/00GQ}{Tag~00GQ}]{stacks-project}. Since $\pi=q\circ r$ is
	surjective, so is~$q$. If $Y$ is integral, then $Z$ is an integral scheme with
	function field~$M\coloneqq k(Y)^{H}$. Indeed, $\mathcal A^{H}$ is then a subalgebra of the constant sheaf~$M$ with generic stalk
	$(\mathcal A\otimes_{\cO_X}k(X))^{H}=k(Y)^{H}=M$, since taking invariants commutes with the flat base change
	to the generic point.

	Fix a big open $U\subseteq X$ such that $\pi_U\colon Y_U\coloneqq\pi^{-1}(U)\to U$ is a
	$G$-torsor, and set $Z_U\coloneqq q^{-1}(U)$. Since $\pi$ is affine and since the $G$-action on $Y_U$ is free, the action of $H$ on $Y_U$ is free and
	its orbits lie in affine opens. 
	By \cite[Exp.~V, Thm.~4.1]{SGA3}, the fppf quotient $Y_U/H$ is the scheme
	$\Spec_U\bigl((\pi_{U*}\cO_{Y_U})^{H}\bigr)=Z_U$ and the restriction $Y_U\to Z_U$ is a finite
	locally free $H$-torsor. 
	Since $\mathcal A^{H}$ is torsion-free over the normal~$X$ and $q$ is
	finite surjective, going-down gives $\dim\cO_{Z,z}=\dim\cO_{X,q(z)}$ for every $z\in Z$. Hence
	$Z_U$ is a big open of~$Z$ and $r$ is a quasi-torsor under~$H$.

If $H$ is normal in $G$, then, by \cite[Exp.~V, Thm.~4.1]{SGA3} applied to the free action of $H$ on~$G$, the fppf quotient $G/H$ is a scheme, affine over~$S$, and a group scheme. Moreover, $G/H$ is finite locally free over~$S$\,: it is flat and locally of finite presentation since $G$ is \cite[\href{https://stacks.math.columbia.edu/tag/06NR}{Tag~06NR}]{stacks-project}, and it is integral since its functions are the $H$-invariant functions on~$G$.
	It acts
	on $Z_U=Y_U/H$ through the action of $G$ on~$Y_U$, and $q$ is a quasi-torsor under $G/H$.
	The isomorphism $Y_U\times_SG\isom Y_U\times_UY_U$,
	$(y,g)\mapsto(y,yg)$, is equivariant for the right multiplication of $H$ on $G$ and its action on
	the second factor, so, the formation of fppf quotients commuting with base change, it induces an
	isomorphism $Y_U\times_S(G/H)\isom Y_U\times_UZ_U$ over~$Y_U$ compatible with the actions
	of~$G/H$. Therefore $Z_U\to U$ becomes the trivial $(G/H)$-torsor after the fppf base change
	$Y_U\to U$, i.e.\ it is a $(G/H)$-torsor.

	Next, let $h\colon X\to S$ be the structure morphism and $\mathcal B\coloneqq h^*\cO(H)$, the finite
	locally free $\cO_X$-algebra with $H\times_SX=\Spec_X\mathcal B$. Then $q_*\cO_Z=\mathcal A^{H}$ is the
	kernel of $\sigma-(\id\otimes1)$ from the reflexive~$\mathcal A$ to the torsion-free $\mathcal A\otimes_{\cO_X}\mathcal B$.
	So it is reflexive
	\cite[\href{https://stacks.math.columbia.edu/tag/0EBG}{Tag~0EBG}]{stacks-project}. Hence the finite surjective~$q$ is a cover. It follows that $Z$ is~$S_2$ by \cref{rmk:cover-reflexive}, and that $q$ is a quasi-torsor cover under $G/H$ when $H$ is normal in~$G$.

	Finally, if $Y$ is normal, then so is~$Z$. Indeed, since normality descends
	along the faithfully flat $Y_U\to Z_U$
	\cite[\href{https://stacks.math.columbia.edu/tag/033G}{Tag~033G}]{stacks-project}, 
	the open $Z_U$ is
	normal. The big open~$Z_U$ contains the points of codimension~$\le1$ of~$Z$, so $Z$ is regular in
	codimension~$\le1$. Since $Z$ is also $S_2$, it is normal by Serre's criterion
	\cite[\href{https://stacks.math.columbia.edu/tag/031S}{Tag~031S}]{stacks-project}.
\end{proof}

\section{Quasi-torsors and trace maps}
\label{sec:q-torsors_upper_shriek}

The trace of a quasi-torsor cover, introduced by Carvajal-Rojas
\cite[\S3]{carvajal_rojas_finite_torsors_over_strongly_f_regular_singularities} over an
algebraically closed field, is the main tool in our proof of \cref{intro:tame}. 
In this section, we provide a direct geometric construction of the trace via fppf descent. As an aside, we prove \cref{thm:KV-mixed}, which 
gives the mixed characteristic analogue of Takayama's theorem~\cite{takayama}.

We fix a finite locally free group scheme $G$ over a scheme $S$
with structure map $\pi_G\colon G\to S$, multiplication $m\colon G\times_SG\to G$ and unit section
$e\colon S\to G$.
The two
projections $G\times_SG\to G$ are denoted $\mathrm{pr}_1$ and $\mathrm{pr}_2$.

\subsection{The relative dualizing sheaf \texorpdfstring{$\omega_{G/S}$}{ω\_G/S} and the module of integrals}\label{sec:omega_K}
The action of~$G$ on itself by left multiplication induces an action on~$\omega_{G/S}$. Explicitly, for any $S$-scheme $T$ and any point $g\in G(T)$, left multiplication $L_g$ by~$g$ acts on 
$\omega_{G_T/T}=\cHom_{\cO_T}(\pi_{G_T*}\cO_{G_T},\cO_T)$ via
$L_g^{*}\colon\mu\mapsto\mu\circ L_g^{\sharp}$, with $L_g^{\sharp}\colon \pi_{G_T*}\cO_{G_T} \to \pi_{G_T*}L_{g*}\cO_{G_T} = \pi_{G_T*}\cO_{G_T}$.

Left multiplication provides a $G$-linearization of~$\omega_{G/S}$, and hence a descent datum for~$\omega_{G/S}$. 
Precisely, left multiplication by
the universal point is the automorphism $\sigma$ of $G\times_SG$
over the first factor, given on points by $(g,h)\mapsto(g,gh)$. 
On the one hand, $\mathrm{pr}_1$ is the base change of $\pi_G$ along~$\pi_G$, which yields a canonical isomorphism
$\omega_{\mathrm{pr}_1}\isom\mathrm{pr}_2^{*}\omega_{G/S}$. 
On the other hand,
since $\mathrm{pr}_1 \circ \sigma = \mathrm{pr}_1$, functoriality of the relative dualizing
sheaf gives
$\sigma^{*}\omega_{\mathrm{pr}_1}\isom\omega_{\mathrm{pr}_1}$. 
Using $\mathrm{pr}_2\circ\sigma=m$, we get
a canonical isomorphism
$\varphi\colon m^{*}\omega_{G/S} = \sigma^* \mathrm{pr}_2^* \omega_{G/S} \isom \sigma^* \omega_{\mathrm{pr}_1} \isom \omega_{\mathrm{pr}_1} \isom \mathrm{pr}_2^{*}\omega_{G/S}$. 
It satisfies the cocycle
condition over $G\times_SG\times_SG$, namely
$(m\times\id)^{*}\varphi=\mathrm{pr}_{23}^{*}\varphi\circ(\id\times m)^{*}\varphi$ 
by
functoriality of the relative dualizing sheaf under composition of left multiplications, and hence defines a $G$-linearization of~$\omega_{G/S}$.
Since $\tau\colon G\times_SG\to G\times_SG$, $(g,h)\mapsto(gh,h)$, is an isomorphism, this $G$-linearization
is the pullback along~$\tau$ of a descent datum relative to~$\pi_G$
(see e.g.\ \cite[Prop.~3.49 \& Thm.~4.46]{vistoli_descent}), 
and descent along $\pi_G$ gives, compatibly with left multiplication, an isomorphism
\[
\omega_{G/S}\isom\pi_G^{*}\!\int_G,\qquad \text{where}\ \int_G\coloneqq e^{*}\omega_{G/S}.
\]
By the projection formula, $\int_G\otimes_{\cO_S}\pi_{G*}\cO_G\isom\pi_{G*}\omega_{G/S}$,
which is finite locally free of the same rank as $\pi_{G*}\cO_G$. As $\id\otimes e^\sharp$ is a
splitting of $\int_G\to\int_G\otimes_{\cO_S}\pi_{G*}\cO_G$, $\lambda\mapsto\lambda\otimes1$, the
$\cO_S$-module~$\int_G$ is a direct summand of a finite locally free module, hence finite
locally free, and comparing ranks shows that it is invertible.
 We call $\int_G$ the
\emph{module of integrals}.
Since the formation of $\omega_{G/S}$ commutes with arbitrary base change $S'\to S$,  so does the formation of $\int_G$.
As $\int_G$ is flat with trivial action, the left-invariant sections of
$\pi_{G*}\omega_{G/S}\isom\int_G\otimes_{\cO_S}\pi_{G*}\cO_G$ are exactly the sections of
$\int_G\otimes_{\cO_S}(\pi_{G*}\cO_G)^G=\int_G$\,; they are called \emph{left integrals}. 
We write $\Tr_G\colon \pi_{G*}\cO_G \to \cO_S$ for a left integral\,; the left-invariance translates to $(\id_G \otimes \Tr_G)\circ\Delta=\Tr_G(-)\cdot 1$, where $\Delta$ is the
comultiplication and $1$ the unit of~$\cO(G)$. 
A left integral generating $\int_G$ exists if and only if
$\int_G\isom\cO_S$, e.g.\ if $S=\Spec R$ with $\Pic(R)=0$, in which case the left integrals form a
free $R$-module of rank one.
This is the
geometric form of the argument of Larson--Sweedler \cite[\S3]{LS} and Pareigis \cite[Lem.~2 \& Prop.~3]{pareigis_when_hopf_algebras_are_frobenius_algebras}.
If a left integral~$\Tr_G$ generates~$\int_G$, then it generates $\omega_{G/S}$ as an
$\cO_G$-module since the isomorphism $\omega_{G/S}\isom\pi_G^*\int_G$ is compatible with left
multiplication. In other words, we have an isomorphism 
\begin{equation}\label{eq:nonsingular-G}
\pi_{G*}\cO_G\stackrel{\simeq}{\longrightarrow} \cHom_{\cO_S}(\pi_{G*}\cO_G,\cO_S), \quad f\mapsto\Tr_G(f\cdot-).
\end{equation}
 Over a principal ideal domain, this is the
nonsingularity of a generating left integral \cite[Thm., pp.~79--80]{LS}, and over a ring~$R$ with
$\Pic(R)=0$, it is \cite[Thm.~1 \& Lem.~3]{pareigis_when_hopf_algebras_are_frobenius_algebras}. 
Note that \cite{LS} and \cite{pareigis_when_hopf_algebras_are_frobenius_algebras} do not assume the Hopf algebra to be commutative\,; the geometric argument above recovers these results only for commutative Hopf algebras, i.e.\ for coordinate rings of finite locally free group schemes.

\subsection{Trace maps associated with quasi-torsor covers}\label{sec:trace_quotient}

We fix a left integral $\Tr_G\colon \pi_{G*}\cO_G \to \cO_S$.
Assume that $G$ acts on a scheme $Y$ over $S$, with action morphism
\[a\colon Y \times_S G \to Y, \quad (y,g) \mapsto yg.\]
In case $S=\Spec R$ is affine, we set
\[
\begin{tikzcd}
\Tr_{Y \times_R G/ Y}
\colon \ \mathrm{pr}_{1*} \cO_{Y \times_R G} = \cO_Y\otimes_R \cO(G) \arrow{rrr}{\id_Y\otimes_R \Tr_G} &&& \cO_Y.
\end{tikzcd}\]
In general, these maps glue to define an $\cO_Y$-linear map
\[\Tr_{Y \times_S G/ Y}
\colon \ \mathrm{pr}_{1*} \cO_{Y \times_S G} \longrightarrow \cO_Y.\]
The following construction is due to Carvajal-Rojas \cite[\S 3]{carvajal_rojas_finite_torsors_over_strongly_f_regular_singularities} for $S=\Spec k$. 
By \cref{rmk:action-extends,rmk:invariants}, it applies to a $G$-torsor $\pi\colon Y\to X$ and to a quasi-torsor cover $\pi\colon Y\to X$ under $G$. 

\begin{defproposition}\label{def-prop:invariant-trace}
	Assume that there exists a $G$-invariant $S$-morphism of schemes $\pi\colon Y \to X$ such that $\cO_X \to (\pi_*\cO_Y)^G$ is an isomorphism.
	The $\cO_X$-linear morphism
	\[
	\begin{tikzcd}
	\pi_*\cO_Y \arrow[r, "\pi_*a^\sharp"] 
	& \pi_*a_* \cO_{Y \times_S G} = \pi_*\mathrm{pr}_{1*} \cO_{Y \times_S G} \arrow[rr, "\pi_*\Tr_{Y\times_SG/Y}"]
	&& \pi_*\cO_Y
	\end{tikzcd}
	\]
	factors through $\cO_X$ to provide an $\cO_X$-linear map  that we denote by 
	\[\Tr_{Y/X} \colon \pi_*\cO_Y \to \cO_X.\]  
\end{defproposition}
\begin{proof} The question is local on $S$ and we may thus assume that $S$ is affine, say $S=\Spec R$.
	We have to show that the $\cO_X$-linear map
	\[
	\begin{tikzcd}
	\pi_*\cO_Y \arrow[r, "\pi_*a^\sharp"]
	& \pi_*\cO_Y \otimes_R \cO(G) \arrow[rr, "\id_Y \otimes_R \Tr_G"] 
	&& \pi_*\cO_Y
	\end{tikzcd}
	\]
	factors through $\cO_X = (\pi_*\cO_Y)^G$\,; i.e.\ that for any local section $s$ of $\pi_*\cO_Y$ and $t \coloneqq (\id_Y\otimes \Tr_G)\circ a^\sharp(s)$ we have $a^\sharp(t) = t\otimes 1$.
	We indeed have
	\begin{align*}
	a^\sharp \circ (\id_Y\otimes \Tr_G)\circ a^\sharp & = (\id_Y \otimes \id_G \otimes\Tr_G) \circ (a^\sharp \otimes \id_G) \circ a^\sharp \\
	& = (\id_Y \otimes \id_G \otimes\Tr_G) \circ  (\id_Y \otimes \Delta) \circ a^\sharp \\
	& = \big((\id_Y \otimes \Tr_G) \circ a^\sharp\big)\otimes 1
	\end{align*}
	where the first equality is obvious, the second is by definition of a coaction, and the third is because $\Tr_G$ is a left integral.  
\end{proof}

Over an algebraically closed field, \cref{lem:trace-generates} is \cite[Cor.~3.13]{carvajal_rojas_finite_torsors_over_strongly_f_regular_singularities}.

\begin{lemma}\label{lem:trace-generates}
	Assume that the left integral $\Tr_G$ generates $\int_G$.
	Let $\pi\colon Y\to X$ be either a $G$-torsor or a quasi-torsor cover under~$G$. Then the
	$\pi_*\cO_Y$-linear map
	\begin{equation}\label{eq:nonsingular}
	\pi_*\cO_Y\longrightarrow\cHom_{\cO_X}(\pi_*\cO_Y,\cO_X),\qquad s\longmapsto\Tr_{Y/X}(s\cdot-),
	\end{equation}
	is an isomorphism. In other words, $\Tr_{Y/X}$ freely generates $\omega_{Y/X}$.
\end{lemma}
\begin{proof}
	Let $\pi$ be a $G$-torsor. Since the formation of $(\pi_*\cO_Y)^G$, 
	of $\Tr_{Y/X}$ and of both sides of the map commutes with flat base change
	on~$X$, and being an isomorphism is fppf-local on~$X$, we may assume that $\pi$ has a section, hence that $Y=X\times_SG$ with $G$ acting by right multiplication on the second factor. In that case, $a^\sharp=\id\otimes\Delta$, so that $\Tr_{Y/X}$ agrees with the map $\Tr_{X\times_SG/X}= \mathrm{pr}_2^*\Tr_G$.
	Hence $\cO_Y\to\omega_{Y/X}$, $1\mapsto\Tr_{Y/X}$, is the pullback along
	$\mathrm{pr}_2\colon Y\to G$ of the isomorphism $\cO_G\to\omega_{G/S}$, $1\mapsto\Tr_G$,
	of \cref{eq:nonsingular-G}.

	In the quasi-torsor cover case, $\pi_*\cO_Y$ is reflexive by definition, and therefore so is
	$\omega_{Y/X} = \cHom_{\cO_X}(\pi_*\cO_Y,\cO_X)$. The map \cref{eq:nonsingular} of reflexive sheaves is an isomorphism over a big open over which
	$\pi$ is a torsor, hence an isomorphism.
\end{proof}

\subsection{Quasi-torsor covers and finite torsors over proper splinters}\label{sec:KV-mixed}
By combining \cref{lem:trace-generates} with \cref{lem:lift}, we obtain the following variant of \cite[Thm.~E]{krah-vial}\,:

\begin{proposition}\label{prop:qtorsor-splinter}
	Assume that $S=\Spec R$ is affine. 
	Let $\pi\colon Y\to X$ be a quasi-torsor cover under a finite locally free $S$-group scheme~$G$ such that either  $\pi^\sharp\colon H^0(X,\cO_X)\to H^0(Y,\cO_Y)$ is an isomorphism or $H^0(Y,\cO_Y)$ is a field.
	\begin{enumerate}
		\item If $X$ is a splinter, then $Y$ is a splinter.
		\item If $S=\Spec k$ and $X$ is a globally $F$-regular variety over $k$, then $Y$ is a globally $F$-regular variety over $k$.
	\end{enumerate}
\end{proposition}
\begin{proof}
	By \cref{lem:lift}, it suffices to show that $\omega_{Y/X}\isom\cO_Y$ in $\Coh(Y)$. 
A Noetherian $R$-scheme is a splinter if and only if 
	its base changes to the localizations $R_\fm$ of~$R$ at its maximal ideals are splinters\,; see \cite[Lem.~7.3]{krah-vial}. Hence, for (i), we may assume $R$ local.
	In both~(i) and~(ii), we thus have $\Pic(R)=0$, and
by \S\ref{sec:omega_K} a left integral $\Tr_G$ generating~$\int_G$ exists. It follows from  
	\cref{lem:trace-generates} that $\Tr_{Y/X}$ freely generates~$\omega_{Y/X}$, in particular that $\omega_{Y/X}\isom\cO_Y$, as desired.
\end{proof}

Over an affine base, the following theorem extends \cite[Thm.~8.2]{krah-vial}, where the base was assumed to be an integral Noetherian scheme of
characteristic~$p$. Its proof follows that of
\cite[Thm.~8.2]{krah-vial}. The vanishing of the positive-degree cohomology of the structure
sheaf, established over a field of characteristic~$p$ in \cite{bhatt_derived_splinters_in_positive_characteristic} (see~\cite[Prop.~3.6]{krah-vial}), is replaced
by \cref{lem:splinter-vanishing} together with
flat base change to the generic fiber.

\begin{theorem}\label{thm:KV-mixed}
	Let $R$ be a Noetherian ring.
	Let $X$ be a proper $R$-scheme whose closed points have residue fields of positive characteristic, and let
	$\pi\colon Y\to X$ be an $R$-morphism such that $\pi^\sharp\colon H^0(X,\cO_X)\to H^0(Y,\cO_Y)$
	is an isomorphism. Assume that $\pi$ is finite \'etale or a torsor under a finite locally free
	$R$-group scheme~$G$. If $X$ is a splinter, then $\pi$ is an isomorphism.
\end{theorem}

\begin{proof}
	Since $X$ is a splinter, it is normal, and 
	working on each connected component of~$X$, we may assume that $X$ is integral. 
	Then $T\coloneqq H^0(X,\cO_X)$ is a domain, and $X$ is proper over~$\Spec T$.
	Moreover, $Y$ is a splinter, as can be seen as follows. If $\pi$ is finite \'etale, then its trace pairing is perfect
	\cite[\href{https://stacks.math.columbia.edu/tag/0BJF}{Tag~0BJF}]{stacks-project}, so
	$\omega_{Y/X}\isom\cO_Y$, and \cref{lem:lift} applies. If $\pi$ is a $G$-torsor, then
 \cref{prop:qtorsor-splinter} applies. 
	
	By \cref{lem:splinter-vanishing}, applied to $X$ and to~$Y$, whose closed points lie over closed points of~$X$ since $\pi$ is finite, $H^i(X,\cO_X)=0=H^i(Y,\cO_Y)$ for all $i>0$.
	Let $K$ be the field of fractions of~$T$, and let
	$\pi_K\colon Y_K\to X_K$ be the base change of~$\pi$ along $\Spec K\to\Spec T$. By flat base
	change \cite[\href{https://stacks.math.columbia.edu/tag/02KH}{Tag~02KH}]{stacks-project},
	$H^0(X_K,\cO_{X_K})=K=H^0(Y_K,\cO_{Y_K})$ and all higher cohomology groups vanish, so that
	$\chi(\cO_{X_K})=1=\chi(\cO_{Y_K})$. Since $X_K$ is proper over~$K$ and $\pi_K$ is finite
	\'etale, resp.\ a torsor under~$G_K$, \cite[Lem.~8.1]{krah-vial} gives
	$\chi(\cO_{Y_K})=\deg(\pi)\,\chi(\cO_{X_K})$. Hence $\deg(\pi)=1$. So $\pi$ is finite locally free of
	rank one, hence an isomorphism.
\end{proof}

\section{Splitting and linear reductivity}\label{sec:tame}

In this section, we prove \cref{thm:tame-base}, which gives the first assertion of
\cref{intro:tame}, see \cref{rmk:affine}. Its proof uses Maschke's theorem over a base, which we
establish first, see \cref{lem:maschke-base}.

\subsection{Linear reductivity and Maschke's theorem}\label{sec:maschke}
Recall that an affine group scheme $G$ over a scheme~$S$ is \emph{linearly reductive} if $V\mapsto V^G$ is exact on quasi-coherent $G$-equivariant $\cO_S$-modules. We have the following characterization of finite locally free linearly reductive group schemes over $S$.

\begin{lemma}[Maschke's theorem over a base]\label{lem:maschke-base}
	Let $G$ be a finite locally free group scheme over~$S$.
	 The following are equivalent\,:
	\begin{enumerate}
		\item $G$ is linearly reductive\,;
		\item $\Tr_s(1)\notin\fm_s$ for every $s\in S$, where 
		$\Tr_s$ is any generator of the invertible
		$\cO_{S,s}$-module $\int_{G\times_S\Spec\cO_{S,s}}$\,;
		\item the evaluation map $\mathrm{ev}_1\colon\int_G\to\cO_S$, $\lambda\mapsto\lambda(1)$, is an
		isomorphism\,;
		\item the unit $\cO_S\to\pi_{G*}\cO_G$ admits a splitting that is a left integral.
	\end{enumerate}
In particular, if $G$ is linearly reductive, then $\int_G$ is free on the unique left integral~$\Tr_G$ with $\Tr_G(1)=1$.
\end{lemma}
\begin{proof} 
(ii)$\Leftrightarrow$(iii). The formation of $\int_G$ commutes with base change, so the stalk of $\int_G$ at~$s$ is
	$\int_{G\times_S\Spec\cO_{S,s}}$. Since $\int_G$ is invertible, this stalk is free of rank one on~$\Tr_s$.
	So the stalk of $\mathrm{ev}_1$ at~$s$ is $\cO_{S,s}\to\cO_{S,s}$, $1\mapsto\Tr_s(1)$, and
	$\mathrm{ev}_1$ is an isomorphism if and only if $\Tr_s(1)\notin\fm_s$ for all~$s$.

	(iii)$\Leftrightarrow$(iv). A left integral~$\lambda$ is a splitting of the unit $\cO_S\to\pi_{G*}\cO_G$ if and only if
	$\mathrm{ev}_1(\lambda)= \lambda(1)=1$. 
	Therefore (iv) holds if and only if $1$ lies in the image of $\mathrm{ev}_1$ on global sections, which is equivalent to (iii) since a map of invertible sheaves is an isomorphism if and only if it is surjective.

	(iv)$\Rightarrow$(i). Let $\Tr_G$ be a left integral with $\Tr_G(1)=1$, and let $V$ be a
	quasi-coherent $G$-equivariant $\cO_S$-module with coaction $a^\sharp\colon V\to V\otimes\pi_{G*}\cO_G$.
As in the proof of \cref{def-prop:invariant-trace}, we see that $R_V\coloneqq(\id_V\otimes\Tr_G)\circ a^\sharp$ maps
	$V$ to~$V^G$. 
	Moreover, $R_V(v)=(\id_V\otimes \Tr_G)(v\otimes 1) = \Tr_G(1)\,v=v$ for all $v\in V^G$. Thus $R_V$ is a projection of $V$ onto $V^G$, natural
		in~$V$, so that 
 taking $G$-invariants is exact.

	(i)$\Rightarrow$(iii). The unit $\cO_S\to\pi_{G*}\cO_G$ is $G$-equivariant and is split by the
	counit~$e^\sharp$. As such, the dual map $\pi_{G*}\omega_{G/S}=\cHom_{\cO_S}(\pi_{G*}\cO_G,\cO_S)\to\cO_S$,
	$\mu\mapsto\mu(1)$, is a surjective $G$-equivariant map of quasi-coherent sheaves.
Taking $G$-invariants and assuming that $G$ is linearly reductive yields that $\mathrm{ev}_1\colon\int_G\to\cO_S$ is
	surjective. Since both sheaves are invertible, it is an isomorphism.
\end{proof}

\subsection{Quasi-torsor covers} We now prove the linear reductivity assertion of \cref{intro:tame}.

\begin{theorem}
	\label{thm:tame-base}
	Let $S$ be a scheme, let $G$ be a finite locally free $S$-group scheme, let $h\colon X\to S$ be an $S$-scheme, and let $\pi\colon Y\to X$ be a quasi-torsor cover under~$G$.
	Assume that $\pi^\sharp\colon h_*\cO_X\to h_*\pi_*\cO_Y$ is an isomorphism. Then the following are
	equivalent\,:
	\begin{enumerate}
		\item the map $\cO_X\to\pi_*\cO_Y$ splits as $\cO_X$-modules\,;
		\item the evaluation map $\mathrm{ev}_1\colon\int_G\to\cO_S$ is an isomorphism over an open
		neighborhood of~$h(X)$\,;
		\item $G\times_SX$ is linearly reductive over~$X$.
	\end{enumerate}
	These conditions hold if $X$ is a splinter. They imply that the fiber $G_s$ is linearly
	reductive for every $s\in h(X)$, and that $G$ is linearly reductive over~$S$ if $h$ is
	surjective.
\end{theorem}
\begin{proof}[Proof of \cref{thm:tame-base}]
	If $X$ is a splinter, then (i) holds, since $\pi$ is finite and surjective.

	(i)$\Rightarrow$(ii). Let $\rho\colon\pi_*\cO_Y\to\cO_X$ be an $\cO_X$-linear splitting of
	$\pi^{\sharp}$. The locus where the morphism $\mathrm{ev}_1$ of invertible sheaves is an
	isomorphism is open. So it suffices to show that $\mathrm{ev}_1$ is an isomorphism at every
	$s\in h(X)$. Let $U\subseteq S$ be an affine open neighborhood of $s$ over which $\int_G$ is
	trivial. The hypotheses are local on~$S$. So we may replace $S$ by~$U$, and $G$, $X$, $Y$,
	$\rho$ by their restrictions over~$U$. Then $\int_G$ is free on a left integral~$\Tr_G$.

	Now $\cO_X\to(\pi_*\cO_Y)^G$ is an isomorphism by \cref{rmk:invariants}, so that the
	trace $\Tr_{Y/X}\colon\pi_*\cO_Y\to\cO_X$ of \cref{def-prop:invariant-trace} is defined, and
	\cref{lem:trace-generates} shows that the $\pi_*\cO_Y$-linear map
	\[
	\pi_*\cO_Y\longrightarrow\cHom_{\cO_X}(\pi_*\cO_Y,\cO_X),\qquad t\longmapsto\Tr_{Y/X}(t\cdot-),
	\]
	is an isomorphism. So $\rho=\Tr_{Y/X}(t\cdot-)$ for a unique $t\in H^0(Y,\cO_Y)$
	with
	$t=\pi^{\sharp}(r)$ for some $r\in H^0(X,\cO_X)$ by hypothesis. As $\Tr_{Y/X}$ is
	$\cO_X$-linear and $\Tr_{Y/X}(1)=\Tr_G(1)$, since $a^\sharp(1)=1\otimes1$, we get
	\[
	1=\rho(1)=\Tr_{Y/X}\bigl(\pi^{\sharp}(r)\bigr)=r\cdot\Tr_{Y/X}(1)=\Tr_G(1)\,r \quad \text{in}\ 	 H^0(X,\cO_X).
	\]
	Hence $\Tr_G(1)\in H^0(S,\cO_S)$ becomes a unit in $\cO_{X,x}$ for every $x\in X$, so that
	$\Tr_G(1)\notin\fm_s$ for~$s=h(x)$. Since $\int_G$ is free on $\Tr_G$, the stalk of
	$\mathrm{ev}_1$ at $s$ is multiplication by $\Tr_G(1)$, hence an isomorphism.

	(ii)$\Rightarrow$(iii). Since the formation of
	$\omega_{G/S}$ and of $\int_G$, hence that of $\mathrm{ev}_1$, commutes with base change,
	$\mathrm{ev}_1$ for $G\times_SX$ is $h^*\mathrm{ev}_1$ and is thus an isomorphism. It follows from \cref{lem:maschke-base} that $G\times_SX$ is
	linearly reductive over~$X$.

	(iii)$\Rightarrow$(i). By  \cref{lem:maschke-base}, there is a left integral $\Tr_{G_X}$ of~$G_X$
	with $\Tr_{G_X}(1)=1$. We regard $\pi$ as a quasi-torsor cover under~$G_X$ over the
	base~$X$. The trace $\Tr_{Y/X}$ defined by $\Tr_{G_X}$ is $\cO_X$-linear with
	$\Tr_{Y/X}(1)=\Tr_{G_X}(1)=1$, hence a splitting of $\cO_X\to\pi_*\cO_Y$.
\end{proof}
\begin{remark}\label{rmk:affine}
	If $S=\Spec R$ is affine and $X$ is Noetherian, then for $f\in R$ the open $X_f\coloneqq h^{-1}(D(f))$ satisfies
	$H^0(X_f,\cO_X)=H^0(X,\cO_X)\otimes_RR_f$ by flat base change, and likewise for~$Y$. It follows that
	$\pi^\sharp\colon h_*\cO_X\to h_*\pi_*\cO_Y$ is an isomorphism if and only if
	$H^0(X,\cO_X)\to H^0(Y,\cO_Y)$ is. So \cref{thm:tame-base}, applied with $S=\Spec R$, gives the
	first assertion of \cref{intro:tame}.
\end{remark}

\section{Cones and traces}\label{sec:cones}

\subsection{Generalities on cones over a base}\label{sec:cones-unif}
Let $R$ be a Noetherian domain, with
field of fractions~$K$.
Let $Y$ be an integral normal scheme, projective over $R$ with $R=H^0(Y,\cO_Y)$, whose generic
fiber $Y_K$ has dimension $e\ge1$, and let $\cL_Y$ be an ample invertible sheaf on~$Y$.
We fix a closed immersion
$Y\hookrightarrow\PP^N_R$ with $\cO(1)|_Y\isom\cL_Y^{q}$ for some $q\ge1$.
Let $P\coloneqq R[z_0,\dots,z_N]$, graded with $z_i$ in degree~$q$, and let
$s_i\in H^0(Y,\cL_Y^{q})$ be the restriction of~$z_i$. Sending $z_i$ to~$s_i$ defines a
graded ring homomorphism from $P$ to \[
A_Y\coloneqq\bigoplus_{m\in \bbZ}H^0(Y,\cL_Y^{m}) = \bigoplus_{m\ge0}H^0(Y,\cL_Y^{m}).
\]
Here the summands with $m<0$
vanish\,; indeed, since $Y$ is integral, $H^0(Y,\cL_Y^{m})$ is a torsion-free $R$-module, and its flat base change
to~$K$ is $H^0(Y_K,\cL_Y^{m}|_{Y_K})$,
which vanishes
for $m<0$ by ampleness of $\cL_Y|_{Y_K}$ on the integral projective $K$-scheme~$Y_K$ of
dimension $e\ge1$.
The ring $A_Y$ is a finite $P$-module
\cite[\href{https://stacks.math.columbia.edu/tag/0B5Q}{Tag~0B5Q}]{stacks-project}, hence Noetherian, with
$(A_Y)_0=H^0(Y,\cO_Y)=R$.
We denote by
\[
c_Y\colon C_Y\coloneqq\Spec_Y\bigl(\bigoplus_{m\in\bbZ}\cL_Y^{m}\bigr)\to Y
\]
the
punctured cone. The group $\Gm$ acts on $C_Y$ over $Y$, by $t^m$
on $\cL_Y^m$, and $c_Y$ is a $\Gm$-torsor, trivial over every open subset of $Y$ over which $\cL_Y$ is trivial. Hence $C_Y$ is integral and normal.
Since $A_Y=H^0(C_Y,\cO_{C_Y})$, the ring $A_Y$ is a normal domain. 
The vertex $V((A_Y)_+)$ of the cone $\Spec A_Y$ has codimension $e+1\ge2$ and
the morphism $C_Y\to\Spec A_Y$ is an open immersion with image
$\Spec A_Y\setminus V((A_Y)_+)$.
Let $f\colon Y'\to Y$ be a finite surjective morphism with $Y'$ normal and
$R=H^0(Y',\cO_{Y'})$. Then $Y'$, with $\cL_{Y'}\coloneqq f^*\cL_Y$, satisfies the same assumptions as~$Y$, and $C_{Y'}=Y'\times_YC_Y$. As in the proof of
\cite[Thm.~7.4]{cai-lee-ma-schwede-tucker},
$A_{Y'}=\bigoplus_{m\ge0}H^0(Y,f_*\cO_{Y'}\otimes\cL_Y^{m})$ is a finite $A_Y$-algebra, by
\cite[\href{https://stacks.math.columbia.edu/tag/0B5Q}{Tag~0B5Q}]{stacks-project} applied to the coherent sheaves
$f_*\cO_{Y'}\otimes\cL_Y^{r}$, $0\le r<q$. Its generic rank is $\deg(f)$ since $C_{Y'}\to C_Y$ has generic degree $\deg(f)$, and $C_Y$ and $C_{Y'}$ are open in
$\Spec A_Y$ and $\Spec A_{Y'}$.

Suppose now that $R$ is local, with maximal ideal $\fm_R$ and residue field~$\kappa$.
Let $M_Y\coloneqq\fm_RA_Y+(A_Y)_+$. Then $A_Y/M_Y=\kappa$.
If $R$ is in addition excellent, then
the completion $\widehat A_Y$ of the local ring $(A_Y)_{M_Y}$ is a Noetherian complete normal local domain with residue
field~$\kappa$ since $(A_Y)_{M_Y}$ is normal and excellent
\cite[\href{https://stacks.math.columbia.edu/tag/07QU}{Tag~07QU} \&
\href{https://stacks.math.columbia.edu/tag/0C23}{Tag~0C23}]{stacks-project}.
Moreover, $M_{Y'}$ is the only prime ideal of $A_{Y'}$
over~$M_Y$. Indeed, $A_{Y'}/M_YA_{Y'}$ is a finite graded $\kappa$-algebra with degree-zero part
$H^0(Y',\cO_{Y'})/\fm_RH^0(Y',\cO_{Y'})=\kappa$, so its homogeneous elements of positive degree are nilpotent. Hence
$(A_Y\setminus M_Y)^{-1}A_{Y'}=(A_{Y'})_{M_{Y'}}$ and
$\widehat A_{Y'}=\widehat A_Y\otimes_{A_Y}A_{Y'}$
\cite[\href{https://stacks.math.columbia.edu/tag/00MA}{Tag~00MA}]{stacks-project}.
Thus $\widehat A_{Y'}$ is finite over~$\widehat A_Y$. By the above applied to~$Y'$, it is a local
domain with residue field~$\kappa$. Finally,
$[\operatorname{Frac}\widehat A_{Y'}:\operatorname{Frac}\widehat A_Y]=\deg(f)$ since the flat base change
$A_Y\to \widehat A_Y$ preserves the generic rank of~$A_{Y'}$.

\subsection{Cones and quasi-torsor covers}
Let $R$ be an excellent Noetherian local domain. Since $\Pic(R)=0$, we may and do  
 fix a left integral $\Tr_G$ that generates $\int_G$ for any finite
 locally free $R$-group scheme~$G$\,; see~\S\ref{sec:omega_K}.
The following lemma is crucial for applying the transformation rule of the $F$-signature~\cite[Thm.~4.8]{carvajal_rojas_finite_torsors_over_strongly_f_regular_singularities}, resp.\ perfectoid signature~\cite[Prop.~6.12]{cai-lee-ma-schwede-tucker}.

\begin{lemma}\label{lem:cone-cover-unif}
	Let $Y$ be an integral normal projective $R$-scheme with $R=H^0(Y,\cO_Y)$, whose generic fiber has positive dimension, let $\cL_Y$ be an ample invertible sheaf on~$Y$, let $f\colon Y'\to Y$ be a
	quasi-torsor cover under a finite locally free $R$-group scheme $G$ with $Y'$
	normal and $R=H^0(Y',\cO_{Y'})$, and let $\cL_{Y'}\coloneqq f^*\cL_Y$.
	\begin{enumerate}
		\item $\Spec A_{Y'}\to\Spec A_Y$ is a quasi-torsor cover under~$G$,
		and its trace $\Tr\colon A_{Y'}\to A_Y$ freely generates
		$\Hom_{A_Y}(A_{Y'},A_Y)$, and is surjective if $G$ is linearly reductive. Hence $\Tr_{\widehat A}\coloneqq\Tr\otimes_{A_Y}\widehat A_Y\colon \widehat A_{Y'}\to \widehat A_Y$
		freely generates $\Hom_{\widehat A_Y}(\widehat A_{Y'},\widehat A_Y)$, and is surjective if $\Tr$~is.
		\item $\Tr_{\widehat A}(\fm_{\widehat A_{Y'}})\subseteq\fm_{\widehat A_Y}$.
	\end{enumerate}
\end{lemma}
\begin{proof}
	(i) The action of $G$ on $Y'$ over $Y$ (\cref{rmk:action-extends}) induces an action on
	$C_{Y'}=Y'\times_YC_Y$ over $C_Y$, hence on $\Spec A_{Y'}$ over
	$\Spec A_Y$. If $f$ is a torsor over a big open $U_0\subseteq Y$, then
	$\Spec A_{Y'}\to\Spec A_Y$  is a torsor over
	the big open $c_Y^{-1}(U_0)$ of $\Spec A_Y$. The finite morphism
	$\Spec A_{Y'}\to\Spec A_Y$ of integral normal schemes is surjective, since its closed image contains
	the dense open $c_Y^{-1}(U_0)$\,; so it is a cover,
and thus a quasi-torsor cover.
	Hence the trace $\Tr$ is defined and freely
	generates $\Hom_{A_Y}(A_{Y'},A_Y)$ by \cref{lem:trace-generates}.
	If $G$ is linearly reductive,
	it is surjective since $\Tr(1)=\Tr_G(1)$ is a unit of~$R$ by \cref{lem:maschke-base}. The
	last assertion follows by flat base change along $A_Y\to \widehat A_Y$.
	
	(ii) The action of $G$ on $C_{Y'}=Y'\times_YC_Y$ in~(i) is through the first factor.
	Since $R$ is local, $\cO(G)$ is free over~$R$, so $H^0(Y,-)$ commutes with $-\otimes_R\cO(G)$.
	Hence the coaction $a^\sharp\colon A_{Y'}\to A_{Y'}\otimes_R\cO(G)$ on
	$A_{Y'}=\bigoplus_{m\ge0}H^0(Y,f_*\cO_{Y'}\otimes\cL_Y^{m})$ is induced by the $\cO_Y$-linear
	coaction on~$f_*\cO_{Y'}$. In particular, it preserves the grading. So the
	$R$-linear map $\Tr=(\id\otimes\Tr_G)\circ a^\sharp$ is homogeneous of degree~$0$. It follows that
	$\Tr(\fm_RA_{Y'}+(A_{Y'})_+)\subseteq\fm_RA_Y+(A_Y)_+$, i.e.\ $\Tr(M_{Y'})\subseteq M_Y$. Since
	$M_{Y'}$ is the only prime ideal of $A_{Y'}$ over~$M_Y$, the maximal ideal of $\widehat A_{Y'}=\widehat A_Y\otimes_{A_Y}A_{Y'}$ is
	$\widehat A_YM_{Y'}$, and we get
	$\Tr_{\widehat A}(\fm_{\widehat A_{Y'}})\subseteq\widehat A_YM_Y=\fm_{\widehat A_Y}$.
\end{proof}

\section{Quasi-torsor cover towers over globally $F$-regular schemes}\label{sec:uniform}
The aim of this section is to prove \cref{thm:main}\,; see \cref{thm:unif} and \cref{thm:main-k}.

\subsection{Statement of main result}\label{sec:unif-statement}

\begin{setup}\label{setup:unif}
	$(R,\fm_R,\kappa)$ is a Noetherian complete local domain whose residue field $\kappa$ is algebraically
	closed of characteristic $p>0$, and $X$ is an integral normal
	scheme, projective over $R=H^0(X,\cO_X)$, whose generic fiber has dimension~$e\ge 1$.
	Moreover, one of the following holds\,:
	\begin{enumerate}
		\itemcase{p}{setup:unif-p} $R$ has characteristic~$p$, and $X$ is globally $F$-regular\,;
		\itemcase{0}{setup:unif-0} $R$ has characteristic~$0$, all fibers of $X\to\Spec R$ have dimension~$e$, and there is a $\bbQ$-divisor $\Delta\ge0$ on $X$
		such that $-(K_X+\Delta)$ is ample and such that for every Cartier divisor $D>0$ on $X$
		there is an $\varepsilon>0$ for which $(X,\Delta+\varepsilon D)$ is globally $+$-regular\,;
	\end{enumerate}
and we fix an ample invertible sheaf $\cL$ on~$X$. In case~\ref{setup:unif-p}, $\cL$ is arbitrary.
In case~\ref{setup:unif-0},
we choose $\cL$ so that \cite[Claim~7.5]{cai-lee-ma-schwede-tucker} applies, namely
$\cL\coloneqq\pi^*\cO(1)\isom\cO_X(-n_0(K_X+\Delta))$ for an integer $n_0\ge1$ and a finite
surjective $R'$-morphism $\pi\colon X\to\PP^e_{R'}$, where $R'\subseteq R$ is a power series ring over
the ring of Witt vectors of~$\kappa$ over which $R$ is finite\,; this is possible, see the proof of \cite[Thm.~7.4]{cai-lee-ma-schwede-tucker}. We set
$A\coloneqq\bigoplus_{m\ge0}H^0(X,\cL^m)$ and $M\coloneqq\fm_RA+A_+$, and we let $\widehat A$ be the
completion of the local ring~$A_M$.
\end{setup}

The scheme $X$ is of finite type over the complete local ring~$R$, hence excellent
\cite[\href{https://stacks.math.columbia.edu/tag/07QW}{Tag~07QW} \&
\href{https://stacks.math.columbia.edu/tag/07QU}{Tag~07QU}]{stacks-project}.
Moreover, $X$ is a splinter. In case~\ref{setup:unif-p}, this holds since $X$ is globally
$F$-regular. In case~\ref{setup:unif-0}, the Cartier divisor $D\coloneqq\operatorname{div}(p)$ is ${}>0$
since the proper dominant morphism $X\to\Spec R$ is surjective. Thus $(X,\Delta+\varepsilon D)$ is
globally $+$-regular for some $\varepsilon>0$, hence so is $(X,0)$
\cite[Lem.~6.7]{bhatt_et_al_globally_+_regular_varieties_and_mmp_for_threefolds_in_mixed_char}, and $X$ is a splinter by
\S\ref{sec:splinters} since it is normal and excellent. 
We also note that every
Noetherian complete local ring $R$ of characteristic~$p$ with residue field~$\kappa$ is a
quotient of a power series ring over~$\kappa$
\cite[\href{https://stacks.math.columbia.edu/tag/032A}{Tag~032A}]{stacks-project}, hence is
$F$-finite if $\kappa$ is perfect. Therefore any scheme $X$ of finite type over such a ring~$R$ is $F$-finite.

	\begin{remark}[Case~\ref{setup:unif-0} is the mixed characteristic analogue of case~\ref{setup:unif-p}]\label{rmk:unif-hyp}
	In case~\ref{setup:unif-p}, the hypothesis that $X$ be globally $F$-regular  holds if and only if there is a $\bbQ$-divisor $\Delta$ as in case~\ref{setup:unif-0}, i.e.\ a $\bbQ$-divisor $\Delta\ge0$ on $X$
	such that $-(K_X+\Delta)$ is ample and such that for every Cartier divisor $D>0$ on $X$
	there is an $\varepsilon>0$ for which $(X,\Delta+\varepsilon D)$ is globally $+$-regular. Indeed, if $(X,\Delta+\varepsilon D)$ is globally $+$-regular, taking $f=F^{e'}$ in the definition shows that $\cO_X\to F^{e'}_*\cO_X(\lfloor p^{e'}(\Delta+\varepsilon D)\rfloor)$ splits, hence so does $\cO_X\to F^{e'}_*\cO_X(D)$ once $p^{e'}\varepsilon\ge1$.
	Since every effective Weil divisor $D'$
	on~$X$ is bounded above by an effective Cartier divisor, namely the divisor of a nonzero section
	of $\cO_X(-D')\otimes\cL^m\subseteq\cL^m$ for $m\gg0$
	\cite[\href{https://stacks.math.columbia.edu/tag/01Q3}{Tag~01Q3}]{stacks-project}, we see that $X$ is globally $F$-regular.
	Conversely, if $X$ is globally $F$-regular, there is $\Delta$ with
	$(X,\Delta)$ globally $F$-regular and $-(K_X+\Delta)$ ample
	\cite[Thm.~4.3(i)]{schwede_smith_globally_f_regular_and_log_Fano_varieties}, the pairs
	$(X,\Delta+\varepsilon D)$ are globally $F$-regular for $\varepsilon$ small
	\cite[Rmk.~3.11]{schwede_smith_globally_f_regular_and_log_Fano_varieties}, and globally
	$F$-regular pairs are globally $+$-regular
	\cite[Lem.~6.14]{bhatt_et_al_globally_+_regular_varieties_and_mmp_for_threefolds_in_mixed_char}.
	
	Whether
	globally $+$-regular schemes of characteristic $p$ are globally $F$-regular is open
	\cite[Rmk.~6.16]{bhatt_et_al_globally_+_regular_varieties_and_mmp_for_threefolds_in_mixed_char}.
	In case~\ref{setup:unif-0}, if $X$ is globally $+$-regular,
	\cite[Conj.~6.17]{bhatt_et_al_globally_+_regular_varieties_and_mmp_for_threefolds_in_mixed_char}
	predicts a $\Delta$ with $-(K_X+\Delta)$ ample and $(X,\Delta)$ globally $+$-regular, but not its
	stability under small perturbations.
\end{remark}

We use the following invariants. In case~\ref{setup:unif-p}, for an $F$-finite Noetherian local
domain~$T$ with residue field~$\kappa$, such as $A_M$ and~$\widehat A$, $s(T)$ is its
$F$-signature \cite{huneke_leuschke}, which exists \cite{tucker_F-signature} and satisfies
$0\le s(T)\le1$ \cite{huneke_leuschke}.
In
case~\ref{setup:unif-0}, for a Noetherian complete local domain~$T$ of characteristic~$0$ with
residue field~$\kappa$ and a system of parameters $\mathbf x=(p,x_2,\dots,x_d)$ of~$T$, $d=\dim T$, we write
$s^{\mathbf x}(T)$ for the
perfectoid signature $s^{\mathbf x}_{\mathrm{perfd}}(T)$
\cite[Not.~4.1 \& Def.~4.8]{cai-lee-ma-schwede-tucker}\,; we have $0\le s^{\mathbf x}(T)\le1$
\cite[Prop.~4.22]{cai-lee-ma-schwede-tucker}, and $s^{\mathbf x}(T)$ may depend on
$\mathbf x$ \cite[Caution~4.2]{cai-lee-ma-schwede-tucker}, although it is conjectured not to
\cite[Conj.~8.1]{cai-lee-ma-schwede-tucker}.

\begin{lemma}\label{lem:positivity-unif}
	In \cref{setup:unif}\,:
	\begin{enumerate}
\item 	In case~\ref{setup:unif-p},
$s(\widehat A) = s(A_M) >0$. 
\item \label{eq:positivity-unif} In
case~\ref{setup:unif-0}, 
$s^{\mathbf x}(\widehat A)>0$ for every system of parameters $\mathbf
x=(p,x_2,\dots,x_d)\text{ of }\widehat A$, where $d\coloneqq\dim \widehat A$.
	\end{enumerate}

\end{lemma}
\begin{proof}
	By \S\ref{sec:cones-unif}, applied to $Y=X$ and $\cL_Y=\cL$, $A$ is a Noetherian domain,
	finitely generated over~$R$, and $\widehat A$ is a Noetherian complete normal local domain with
	residue field~$\kappa$.
	In case~\ref{setup:unif-p}, $A$ is thus $F$-finite since $R$ is, and
	$s(\widehat A)=s(A_M)$ by
	\cite[Thm.~5.6]{yao_F-signature} applied to the flat local homomorphism $A_M\to \widehat A$, whose
	closed fiber is the field~$\kappa$. For an invertible sheaf $\cL'$
	on $X$ and a nonzero $a\in H^0(X,\cL')$, global $F$-regularity of~$X$, applied to the divisor
	of~$a$, gives $e'>0$ such that the composite
	$\cO_X\to F^{e'}_*\cO_X\xrightarrow{\,a\,}F^{e'}_*\cL'$ splits\,; so $A$ is quasi-$F$-regular
	\cite[Thm.~2.6, (2)$\Rightarrow$(4), and (2.5)]{hashimoto_surjectivity_multiplication}, i.e.\ strongly $F$-regular
	\cite[(2.1)]{hashimoto_surjectivity_multiplication}, and so is $A_M$
	\cite[Thm.~3.1]{hochster_huneke_tight_closure_and_strong_f_regularity}. Hence $s(A_M)>0$
	\cite[Thm.~0.2]{aberbach_leuschke_signature}. (If $R$ is a field, the strong $F$-regularity
	of~$A$ is \cite[Prop.~5.3]{schwede_smith_globally_f_regular_and_log_Fano_varieties}.)
	In case~\ref{setup:unif-0}, let $\Delta_{\widehat A}$ be the
	$\bbQ$-divisor on $\Spec \widehat A$ induced by~$\Delta$\,; then \ref{eq:positivity-unif} follows
	from \cite[Claim~7.5]{cai-lee-ma-schwede-tucker} and
	\cite[Prop.~5.8]{cai-lee-ma-schwede-tucker}.
\end{proof}

\Cref{thm:unif} below establishes \cref{thm:main}.

\begin{theorem}\label{thm:unif}
	In \cref{setup:unif}, consider a sequence of finite surjective morphisms of $S_2$ schemes whose irreducible components all dominate~$X$
	\[
	X=Y_0\xleftarrow{\ f_0\ } Y_1\xleftarrow{\ f_1\ } Y_2\xleftarrow{\ f_2\ }\cdots,
	\]
	where each $f_i$ is a quasi-torsor under a finite locally free $R$-group scheme $G_i$ and
	where $R=H^0(Y_i,\cO_{Y_i})$ for all~$i$. Then each $Y_i$ is an integral splinter, each $G_i$ is linearly reductive over~$R$, and the sequence stabilizes. Moreover, for all~$n$, we have
	\begin{enumerate}
	\item 	$\deg(Y_n\to X)\le1/s(\widehat A)$ in case~\ref{setup:unif-p}\,;
	\item $\deg(Y_n\to X)\le1/s^{\mathbf x}(\widehat A)$ in case~\ref{setup:unif-0},
		for a system of parameters $\mathbf x=(p,x_2,\dots,x_d)$ of $\widehat A$ depending on the sequence.
	\end{enumerate}
\end{theorem}

\begin{remark}\label{rmk:unif-clmst}
	Suppose that all~$G_i$ are constant, so that all~$f_i$ are quasi-\'etale. In
	case~\ref{setup:unif-p} with $R$ a field, the stabilization could alternatively be deduced, after
	passing to Galois closures, from \cite[Cor.~8.5 \& Thm.~D]{krah-vial} and
	\cite[Main~Thm.]{bhatt_carvajal-rojas_graf_schwede_tucker}. In case~\ref{setup:unif-0}, the
	stabilization and the bound, even with $\mathbf x$ independent of the sequence, follow from the
	proof of \cite[Thm.~7.4]{cai-lee-ma-schwede-tucker}, which applies
	\cite[Thm.~6.1]{cai-lee-ma-schwede-tucker} valid for every system of parameters. However, in general, the~$f_i$
	need not be quasi-\'etale, e.g.\ if $G_i=\mu_{p,R}$, which is ramified over~$p$, and
	\cite[Thm.~6.1]{cai-lee-ma-schwede-tucker} does not apply. We use instead
	\cite[Prop.~6.12]{cai-lee-ma-schwede-tucker}. It requires a system of parameters adapted to all
	the~$f_i$, which Step~4 of \S\ref{sec:proof-unif} constructs. This is why $\mathbf x$ depends on
	the sequence. 
\end{remark}

\subsection{Proof of \texorpdfstring{\cref{thm:unif}}{the theorem}}\label{sec:proof-unif}  
We proceed in several steps.

\emph{Step 1\,: splinters.} We show by induction on $i$ that $Y_i$ is a splinter, hence normal.
For $i=0$, this was noted after \cref{setup:unif}. If $Y_i$ is a splinter, then it is normal, and connected as $H^0(Y_i,\cO_{Y_i})=R$ is a domain, hence integral. Since the irreducible components of~$Y_{i+1}$ dominate~$X$, they dominate~$Y_i$, which is integral and finite over~$X$.
Therefore $f_i$ is a quasi-torsor cover, and $Y_{i+1}$ is a splinter by \cref{intro:tame} since
$H^0(Y_i,\cO_{Y_i})=R=H^0(Y_{i+1},\cO_{Y_{i+1}})$. 

\emph{Step 2\,: linear reductivity.} By Step~1, \cref{intro:tame} applies to the quasi-torsor
cover~$f_i$ of the splinter~$Y_i$ and, since the structure morphism $Y_i\to\Spec R$ is surjective, $G_i$ is linearly reductive over~$R$.

\emph{Step 3\,: cones.} With $\cL_{Y_i}$ the pullback of~$\cL$, each $Y_i$ is as in
\S\ref{sec:cones-unif}.
We write $A_i\coloneqq A_{Y_i}$ and $\widehat A_i\coloneqq \widehat A_{Y_i}$, so that
$A_0=A$ and $\widehat A_0=\widehat A$. Thus
$\widehat A_i\subseteq \widehat A_{i+1}$ is a finite extension of Noetherian complete normal local domains with
residue field~$\kappa$, of generic rank $\deg(f_i)$. Since $G_i$ is linearly reductive by Step~2,
\cref{lem:cone-cover-unif}, applied to~$f_i$, shows that the trace $\Phi_i\coloneqq\Tr_{\widehat A}\colon \widehat A_{i+1}\to \widehat A_i$ freely
generates $\Hom_{\widehat A_i}(\widehat A_{i+1},\widehat A_i)$ as an $\widehat A_{i+1}$-module, is surjective, and satisfies
$\Phi_i(\fm_{\widehat A_{i+1}})\subseteq\fm_{\widehat A_i}$.
These are the hypotheses on~$T$ of
\cite[Thm.~4.8]{carvajal_rojas_finite_torsors_over_strongly_f_regular_singularities} and the
conditions (a)--(c) of \cite[Lem.~6.11]{cai-lee-ma-schwede-tucker}.

\emph{Step 4\,: a common system of parameters, in case~\ref{setup:unif-0}.} In order to apply
\cite[Prop.~6.12]{cai-lee-ma-schwede-tucker} in the next step, we choose, for each~$i$, an
element~$g_i$ such that $\widehat A_i[1/g_i]\subseteq\widehat A_{i+1}[1/g_i]$ is finite free and finite \'etale and, as in the proof of
\cite[Thm.~7.12]{cai-lee-ma-schwede-tucker}, a system of parameters $\mathbf x$ adapted to all
the~$g_i$ (see \cref{rmk:unif-clmst}). Recall that $d=\dim \widehat A$. By Cohen's structure theorem
\cite[\href{https://stacks.math.columbia.edu/tag/032D}{Tag~032D}]{stacks-project}, there is a
subring $B\subseteq \widehat A$ over which $\widehat A$ is finite, with $B$ a power series ring over the ring of
Witt vectors of~$\kappa$, i.e.\ a complete unramified regular local ring of dimension~$d$ with
residue field~$\kappa$. Every $\widehat A_i$ is then finite over~$B$, with the same residue field~$\kappa$. For
each~$i$, there is $g_i\in\fm_B\setminus\{0\}$ such that $B[1/g_i]\subseteq \widehat A_i[1/g_i]$ and
$\widehat A_i[1/g_i]\subseteq \widehat A_{i+1}[1/g_i]$ are finite free and finite \'etale. Indeed, by generic freeness
\cite[\href{https://stacks.math.columbia.edu/tag/051R}{Tag~051R}]{stacks-project}, these extensions become
finite free after inverting a nonzero element of~$B$, resp.\ of~$\widehat A_i$. The discriminants
of the resulting finite free extensions, i.e.\ the determinants of their trace pairings, are nonzero since the
extensions of fields of fractions have characteristic zero, hence are separable. Inverting
the discriminants makes the extensions \'etale
\cite[\href{https://stacks.math.columbia.edu/tag/0BJF}{Tag~0BJF}]{stacks-project}. Finally,
it suffices to invert a nonzero element of~$B$, which we multiply by~$p$ to place it in~$\fm_B$.
Indeed, for $0\neq b\in B$, a nonzero element of $\widehat A_i[1/b]$ is integral over~$B[1/b]$, hence
divides a nonzero element of~$B$. As in the proof of
\cite[Claim~7.13]{cai-lee-ma-schwede-tucker}, \cite[Lem.~7.11]{cai-lee-ma-schwede-tucker} gives
successively $x_2,\dots,x_d\in\fm_B\setminus\fm_B^2$, $x_j$ avoiding the countably many minimal
primes of the ideals $(g_i,x_2,\dots,x_{j-1})$, $i\ge0$, none of which is~$\fm_B$ by Krull's
height theorem. Then $\dim B/(g_i,x_2,\dots,x_j)\le d-j$ for all~$i$ and $1\le j\le d$. Indeed,
$\dim B/(g_i)=d-1$ for the nonzero element $g_i$ of the domain~$B$, and every minimal prime of
$(g_i,x_2,\dots,x_j)$ strictly contains one of $(g_i,x_2,\dots,x_{j-1})$, which $x_j$ avoids.
Hence $g_i,x_2,\dots,x_d$ is a system of parameters of~$B$ for every~$i$. Since the proof of
\cite[Lem.~7.11]{cai-lee-ma-schwede-tucker} allows one to prescribe the class of $x_j$ in $\fm_B/\fm_B^2$, we may moreover take $p,x_2,\dots,x_d$ to be a
regular system of parameters of~$B$ in a way that $B$ is the power series ring in $x_2,\dots,x_d$ over
the ring of Witt vectors of~$\kappa$, as is required in \cite[Prop.~6.12]{cai-lee-ma-schwede-tucker}.
 Set $\mathbf x\coloneqq(p,x_2,\dots,x_d)$ to be a system of parameters of~$B$, and hence 
of every~$\widehat A_i$.

\emph{Step 5\,: conclusion.} In case~\ref{setup:unif-p}, the rings $\widehat A_i$ are $F$-finite. Applying
\cite[Thm.~4.8]{carvajal_rojas_finite_torsors_over_strongly_f_regular_singularities} 
gives $s(\widehat A_{i+1})=\deg(f_i)\cdot s(\widehat A_i)$.
 In case~\ref{setup:unif-0}, applying \cite[Prop.~6.12]{cai-lee-ma-schwede-tucker} 
 gives
$s^{\mathbf x}(\widehat A_{i+1})=\deg(f_i)\cdot s^{\mathbf x}(\widehat A_i)$. 
In case~\ref{setup:unif-p}, we obtain
\[
\Bigl(\prod_{i<n}\deg(f_i)\Bigr)\,s(\widehat A)=s(\widehat A_n)\le1
\]
for all~$n$ and, since $s(\widehat A)>0$ by \cref{lem:positivity-unif}, we get
$\deg(Y_n\to X)=\prod_{i<n}\deg(f_i)\le1/s(\widehat A)$ for all~$n$. Similarly with $s^{\mathbf x}$ in place of $s$, we obtain $\deg(Y_n\to X)=\prod_{i<n}\deg(f_i)\le1/s^{\mathbf x}(\widehat A)$ in case~\ref{setup:unif-0}.
 Each
nontrivial step contributes a factor $\ge2$, so $\deg(f_n)=1$ for all large~$n$, and $f_n$ is then
a finite birational morphism to the normal~$Y_n$, hence an isomorphism.\qed

\subsection{Globally $F$-regular varieties over a field}\label{sec:proof-main-unif}
We now provide a version of \cref{thm:main} for globally $F$-regular varieties over a field that is not necessarily algebraically closed.

\begin{theorem}
	\label{thm:main-k}
	Let $X$ be a globally $F$-regular projective variety over~$k$. 
	Consider a
	sequence of finite surjective morphisms of $S_2$ schemes whose irreducible components all dominate~$X$
	\[
	X=Y_0\xleftarrow{\ f_0\ } Y_1\xleftarrow{\ f_1\ } Y_2\xleftarrow{\ f_2\ }\cdots,
	\]
	where each $f_i$ is a quasi-torsor under some finite $k$-group scheme $G_i$ and where
	$H^0(X,\cO_X)=H^0(Y_i,\cO_{Y_i})$ for all~$i$. Then each $G_i$ is linearly reductive and the
	sequence stabilizes.
	Moreover, $\deg(Y_n\to X)\le 1/s(A_M)$ for all~$n$, where $A$ is the section ring of any ample
	invertible sheaf on $X\times_L\bar k$, for $L\coloneqq H^0(X,\cO_X)$ embedded into~$\bar k$,
	$M\coloneqq A_+$ is its homogeneous maximal ideal, and $s(A_M)$ is the $F$-signature of the
	local ring~$A_M$.
\end{theorem}

\begin{proof}
If $Y'\to Y$ is a
	quasi-torsor cover of a globally $F$-regular projective variety $Y$ over $k$ with
	$H^0(Y,\cO_Y)=H^0(Y',\cO_{Y'})$, then $Y'$ is normal and globally $F$-regular by
	\cref{prop:qtorsor-splinter}.
	Moreover, since $Y$ is integral,  $H^0(Y,\cO_{Y})=H^0(Y',\cO_{Y'})$ is a domain, so $Y'$ is connected, hence integral. 
	By induction on $i$, every $Y_i$ is thus normal and
globally $F$-regular.
Since the irreducible components of~$Y_{i+1}$ dominate~$X$, they dominate~$Y_i$, which is integral and finite over~$X$\,; so every $f_i$ is a quasi-torsor cover.
 In particular, \cref{intro:tame} applies with $R=k$ to the
quasi-torsor cover~$f_i$ of the splinter~$Y_i$ and so $G_i$ is linearly reductive over~$k$. 
The domain $L\coloneqq H^0(X,\cO_X)$ is
a finite $k$-algebra since $X$ is proper, hence a finite field extension of~$k$. Every $Y_i$ is a
projective $L$-scheme with $H^0(Y_i,\cO_{Y_i})=L$, and $Y_{i+1}\times_kG_i=Y_{i+1}\times_LG_{i,L}$,
so $f_i$ is a quasi-torsor under $G_{i,L}$. Replacing $k$ by~$L$, we may assume
$H^0(X,\cO_X)=H^0(Y_i,\cO_{Y_i})=k$ for every~$i$. Since $\bar k$ is $F$-finite, each
$Y_{i,\bar k}\coloneqq Y_i\times_k\bar k$ is normal and globally $F$-regular by
\cite[Cor.~2.8]{gongyo_li_patakfalvi_zsolt_schwede_tanaka_zong_on_rational_connectedness_of_globally_f_regular_threefolds}\,;
see \cite[Prop.~6.9]{krah-vial}. It is also integral since
$H^0(Y_{i,\bar k},\cO_{Y_{i,\bar k}})=\bar k$. 
Moreover, the morphism $f_{i,\bar k}$ 
is a quasi-torsor cover under~$G_{i,\bar k}$.
 If $\dim X\ge1$, then $R\coloneqq\bar k$
and $X_{\bar k}$ are as in case~\ref{setup:unif-p} of \cref{setup:unif}, with $e=\dim X$, and, for any choice of the ample invertible
sheaf $\cL$ on~$X_{\bar k}$, \cref{thm:unif} applies to the sequence $(Y_{i,\bar k})_i$\,: $f_{i,\bar k}$ is an isomorphism for
all large~$i$, and
$\deg(Y_n\to X) = \deg(Y_{n,\bar k}\to X_{\bar k})\le1/s(\widehat A)=1/s(A_M)$ for all~$n$ by \cref{lem:positivity-unif},
where $A$ is the section ring of~$\cL$ and $M=A_+$ its homogeneous maximal ideal. If $\dim X=0$, then
$X_{\bar k}$ is integral and finite over~$\bar k$, so $X_{\bar k}=\Spec\bar k$. Thus each
$f_{i,\bar k}$ is an isomorphism and we get $1 = \deg(Y_n\to X)\le 1/s(A_M)=1$.
\end{proof}

\begin{example}\label{ex:local-tower}
	Let $\mu_p$ act on $\bar k[[u,v]]$ with weights $(1,-1)$.
	The invariant ring is the $A_{p-1}$-singularity
	$\bar k[[u,v]]^{\mu_p}=\bar k[[u^p,v^p,uv]]$. In the tower
	$\bar k[[u^p,v^p]]\subset\bar k[[u,v]]^{\mu_p}\subset\bar k[[u,v]]$, the first inclusion is a
	$\mu_p$-torsor, namely the Kummer cover attached to the unit $1+u^pv^p$, with $p$-th root
	$y=1+uv$,
	and the
	second is the $\mu_p$-quasi-torsor cover for the weight-$(1,-1)$ action, branched at the closed
	point. Iterating gives an infinite tower of
	$\mu_p$-quasi-torsor covers of $\Spec\bar k[[u^p,v^p]]$ with strongly $F$-regular members and
	infinitely many branched steps. The
	transformation rule fails at the Kummer step\,: $s(\bar k[[u^p,v^p]])=1$ whereas $s(\bar k[[u^p,v^p,uv]])=s(\bar k[[u,v]])/p=1/p$ by \cite[Prop.~4.20]{carvajal_rojas_finite_torsors_over_strongly_f_regular_singularities}. 
	The trace of the Kummer step maps the element $y-1=uv$ of the maximal ideal to a unit
	\cite[Ex.~4.10]{carvajal_rojas_finite_torsors_over_strongly_f_regular_singularities}, meaning that the inclusion of
	\cref{lem:cone-cover-unif}(ii), which relies on the global grading, fails for this local torsor.
\end{example}

\section{The Nori fundamental group scheme}
\label{sec:Nori}

We first extract the relevant properties of the Nori fundamental group scheme
\cite{nori_the_fundamental_group_scheme} from the gerbe formulation of
Borne--Vistoli~\cite{borne_vistoli_nori_fundamental_gerbe}. 

\begin{lemma} \label{lem:nori-torsor}
	Let $X$ be a geometrically integral scheme of finite type over $k$, and let
	$U\subseteq X$ be a nonempty open subset. Then
	$U$ carries a Nori fundamental gerbe in the sense of
	\cite[Def.~5.1]{borne_vistoli_nori_fundamental_gerbe}. In particular, for any
	$x\in U(k)$ the Nori fundamental group scheme $\piN(U,x)$
	\cite[Rmk.~5.15]{borne_vistoli_nori_fundamental_gerbe} is a profinite $k$-group scheme,
	and, for every finite $k$-group scheme~$G$, homomorphisms $q\colon\piN(U,x)\to G$ correspond
	bijectively to isomorphism classes of pointed $G$-torsors $V\to U$.

	If moreover $X$ is proper and $U$ is a big open of $X$, then $H^0(V,\cO_V)=k$ for every such
	pointed $G$-torsor $V\to U$ with $q$ faithfully flat.
\end{lemma}
\begin{proof}
	Write $j\colon U\hookrightarrow X$ for the inclusion. 
	Since $X$ is geometrically integral,
	$U$ is also geometrically integral. In particular, $U$ is geometrically connected and geometrically reduced, hence
	inflexible \cite[Prop.~5.5(b)]{borne_vistoli_nori_fundamental_gerbe}, and it thus carries a Nori
	fundamental gerbe \cite[Cor.~5.8]{borne_vistoli_nori_fundamental_gerbe}.

	Fixing $x\in U(k)$ gives the profinite
	$k$-group scheme $\piN(U,x)$
	\cite[Rmk.~5.15]{borne_vistoli_nori_fundamental_gerbe}. The bijection of the statement is
	\cite[Rmk.~5.15]{borne_vistoli_nori_fundamental_gerbe} and
	\cite[Rmk.~1.4]{antei_biswas_emsalem_tonini_zhang}. It sends a homomorphism
	$q\colon\piN(U,x)\to G$ to the pointed $G$-torsor $V\to U$ whose classifying morphism
	$\varphi\colon U\to \BG G$ is the composite of the
	morphism from $U$ to its Nori fundamental gerbe $\BG\piN(U,x)$ with the morphism
	$\BG\piN(U,x)\to\BG G$ induced by~$q$. Here $\BG G$ is the classifying stack of~$G$.

	Finally, suppose $X$ is proper and $U\subseteq X$ is a big open, and let $V\to U$ and
	$\varphi$ be as above, with $q$ faithfully flat. Then $U$ is pseudo-proper
	\cite[Def.~7.1]{borne_vistoli_nori_fundamental_gerbe}, which for a quasi-compact scheme $U$
	amounts to $H^0(U,E)=H^0(X,j_*E)$ being finite-dimensional for every vector bundle $E$ on $U$. Since $X$ is proper,
	the latter follows from the coherence of $j_*E$ on~$X$, which is guaranteed by
	\cite[\href{https://stacks.math.columbia.edu/tag/0AWA}{Tag~0AWA}]{stacks-project}.
	As $U$ is inflexible and pseudo-proper and $q$ is a quotient, $\varphi$ is Nori-reduced
	\cite[Def.~5.10]{borne_vistoli_nori_fundamental_gerbe} and
	$\varphi_*\cO_U=\cO_{\BG G}$ \cite[Rmk.~1.14]{antei_biswas_emsalem_tonini_zhang}. Flat base
	change along the faithfully flat atlas $\Spec k\to \BG G$, whose pullback of $\varphi$ is $V$,
	identifies the pullback of $\varphi_*\cO_U$ with $H^0(V,\cO_V)$, and we obtain $H^0(V,\cO_V)=k$.
\end{proof}

Since the regular locus of a normal variety is a big open,
the following theorem contains \cref{thm:piN}.

\begin{theorem}
	\label{thm:piN2}
	Let $X$ be a globally $F$-regular projective variety over~$k$, let $U\subseteq X$ be a big
	open, and let $x\in U(k)$. Then $\piN(U,x)$ is a finite linearly reductive $k$-group scheme.
	Moreover, its
	order satisfies $|\piN(U,x)|\le 1/s(A_M)$, where $A$ is the section ring of any ample invertible
	sheaf on $X\times_k\bar k$, $M\coloneqq A_+$ is its homogeneous maximal ideal, and $s(A_M)$ is the
	$F$-signature of the local ring~$A_M$.
\end{theorem}

\begin{proof}
	As $X$ is globally $F$-regular projective, it is a proper splinter. 
	Since moreover $X(k)\neq \varnothing$, $H^0(X,\cO_X)= k$.
	Let us check the hypothesis of \cref{lem:nori-torsor}. By \cite[Cor.~5.16]{krah-vial}, the
	connected proper splinter $X$ is geometrically normal over $H^0(X,\cO_X)=k$. Since
	$H^0(X\times_k\bar k,\cO_{X\times_k\bar k})=H^0(X,\cO_X)\otimes_k\bar k=\bar k$ by flat base change, we get that
	$X\times_k\bar k$ is connected. This gives that $X$ is
	geometrically integral.
	 By
	\cref{lem:nori-torsor}, $\piN(U,x)$ is then a profinite $k$-group scheme and every
	finite quotient $\piN(U,x)\twoheadrightarrow G$ is realized by a pointed
	$G$-torsor $V\to U$ with $H^0(V,\cO_V)=k$.

	Fix  a finite quotient $\piN(U,x)\twoheadrightarrow G$ and denote $V\to U$ the associated pointed $G$-torsor. Over the normal base~$X$,
	\cref{prop:quasitorsor-correspondence} produces a quasi-torsor cover
	$Y\to X$ under~$G$ with $Y|_U=V$ and
	$H^0(Y,\cO_Y)=H^0(V,\cO_V)=k=H^0(X,\cO_X)$. By \cref{intro:tame}, $G$ is linearly reductive\,; and \cref{thm:main-k}, applied to the
	one-step tower $X\leftarrow Y$, bounds the order\,:
	\[
	|G|\ =\ \deg(Y\to X)\ \le\ 1/s(A_M),
	\]
	with $A$ and $M$ as in \cref{thm:piN2}.

	Thus every finite quotient of $\piN(U,x)$ is linearly reductive of order at most $1/s(A_M)$.
	Since $\piN(U,x)$ is profinite,  $\piN(U,x)$ is then finite with order at most $1/s(A_M)$.
  It is linearly reductive as it is one of its own finite quotients.
\end{proof}

\section{Quasi-torsors under \texorpdfstring{$\alpha_p$}{α\_p} and \texorpdfstring{$\ZZ/p$}{Z/p} over splinters}
\label{sec:second-proof}

The aim of this final section is to give an alternative proof of the first assertion of \cref{intro:tame} for $R=k$. This proof is closer in spirit to \cite{krah-vial}. It starts by establishing that a splinter $X$ over $\bbF_p$ carries no quasi-torsor cover $Y\to X$ under $\alpha_p$ or $\ZZ/p$ with $H^0(X,\cO_X)=H^0(Y,\cO_Y)$. The following lemma in fact establishes slightly more.

\begin{lemma}\label{lem:Z_p_alpha_p}
	Let $X$ be an integral normal locally Noetherian $\bbF_p$-scheme, let $G$ be either $\alpha_p$ or
	$\ZZ/p$, and let $\pi\colon Y\to X$ be a quasi-torsor cover under~$G$. If the local rings of~$X$
	are splinters, then $\pi$ is a torsor under~$G$. 
	If $X$ is a splinter, then there exists $a \in H^0(X,\cO_X)$ such that
	\begin{enumerate}
		\item $Y \isom \Spec_X \cO_X[T]/(T^p-a)$ over $X$ if $G=\alpha_p$\,;
		\item $Y\isom \Spec_X \cO_X[T]/(T^p-T-a)$ over $X$ if $G=\ZZ/p$\,;
	\end{enumerate}
	in particular, $H^0(Y,\cO_Y)$ is then free of rank~$p$ over $H^0(X,\cO_X)$, so that
	$\pi^\sharp\colon H^0(X,\cO_X)\to H^0(Y,\cO_Y)$ is not an isomorphism.
\end{lemma}
\begin{proof}
	As $\pi$ is finite, it is a torsor if and only if its base changes to the local schemes
	$\Spec\cO_{X,x}$, $x\in X$, are torsors, and these base changes are again quasi-torsor covers
	under~$G$ since localization preserves reflexivity and big opens. So we may assume that
	$X$ is a splinter. Write $P_a\coloneqq\Spec_X\cO_X[T]/(f_a)$ with
	$f_a\coloneqq T^p-a$ if $G=\alpha_p$ and $f_a\coloneqq T^p-T-a$ if $G=\ZZ/p$. 	The Frobenius sequence
	$0\to\alpha_p\to\mathbb G_a\xrightarrow{F}\mathbb G_a\to0$ and the Artin--Schreier sequence
	$0\to\ZZ/p\to\mathbb G_a\xrightarrow{F-1}\mathbb G_a\to0$
	 give an exact
	sequence $H^0(X,\cO_X)\xrightarrow{P} H^1_{\mathrm{fppf}}(X,G)\xrightarrow{\iota}H^1(X,\cO_X)$, where $P$ maps $a$ to the class of~$P_a$.

	Assume first that $\pi$ is a torsor, with class~$c$ in $H^1_{\mathrm{fppf}}(X,G)$. 
	As $X$ is a splinter, $\cO_X\to\pi_*\cO_Y$
	splits, and it follows that $H^1(X,\cO_X)\to H^1(X,\pi_*\cO_Y)=H^1(Y,\cO_Y)$ is injective. Since the pullback of $\pi$ along itself is the
	trivial torsor, we have $\pi^*\iota(c)=\iota(\pi^*c)=0$. It follows that $\iota(c)=0$ and hence  that $Y\isom P_a$ for some~$a$.

	In general, let $U\subseteq X$ be a big open over which $\pi$ is a torsor. As $U$ is a nonempty
	splinter, we get $\pi^{-1}(U)\isom P_a|_U$ for some
	$a\in H^0(U,\cO_U)=H^0(X,\cO_X)$. Since $\pi_*\cO_Y$ and $\cO_X[T]/(f_a)$ are reflexive algebras
	with $G$-coaction which agree when restricted to~$U$, \cref{lem:reflexive-coaction} gives $Y\isom P_a$.
\end{proof}

\begin{remark}
	Following \cite[Claims~4.15 \& 4.16]{carvajal_rojas_finite_torsors_over_strongly_f_regular_singularities}, one can show that any quasi-torsor cover $\pi\colon Y \to X$ under $G$ is a torsor under $G$ in the following cases\,:
	\begin{enumerate}
		\item if $G=\alpha_p$ and $X$ is a normal $F$-injective Noetherian scheme over $k$\,;
		\item if $G=\ZZ/p$ and $X$ is an $F$-rational excellent Noetherian scheme over $k$.
	\end{enumerate}
\end{remark}

We use the following characterization of linearly reductive finite group schemes.

\begin{lemma}\label{lem:group_mult}
	Let $G$ be a finite group scheme over $k$. Then $G$ is linearly reductive if and only if $G_{\bar{k}}$ contains no subgroup scheme isomorphic to $\alpha_p$ and $p\nmid|\pi_0(G)|$.  
\end{lemma}
\begin{proof}
	This is a combination of Nagata's theorem \cite[IV, \S3, Thm.~3.6]{demazure_gabriel_groupes_algebriques}, stating that an affine group scheme over $k$ is linearly reductive if
	and only if the neutral component $G^\circ$ is of multiplicative type and the \'etale part
	$\pi_0(G)\coloneqq G/G^\circ$ has order prime to $p$,  and of \cite[IV, \S3, Lem.~3.7]{demazure_gabriel_groupes_algebriques} stating that
	$G^{\circ}$ is of multiplicative type if and only if
	$G_{\bar{k}}$ contains no subgroup scheme isomorphic to~$\alpha_p$.
\end{proof}

\begin{lemma}\label{lem:crepant-nochar}
	Let $X$ be an integral normal locally Noetherian $k$-scheme, let $\pi\colon Y\to X$ be a
	quasi-torsor cover under a finite $k$-group scheme $G$ with $H^0(Y,\cO_Y)=k$, and let
	$H\subseteq G$ be a closed subgroup scheme with no nontrivial character over~$k$, e.g.\
	$H=\alpha_p$ or $H=\ZZ/p$. Then $q\colon Z=Y/H\to X$ is \emph{crepant}, i.e.\
	$\omega_{Z/X}\isom\cO_Z$ in $\Coh(Z)$. If moreover  $X$ is a splinter,
	resp.\ a globally $F$-regular variety, then so is~$Z$.
\end{lemma}
\begin{proof}
	By \cref{lem:crepancy}, there is a big open $U\subseteq X$ such that $\pi$ restricts to a $G$-torsor
	$Y_U\coloneqq\pi^{-1}(U)\to U$, which factors as $Y_U\xrightarrow{\,r\,}Z_U\xrightarrow{\,q\,}U$ with
	$Z_U\coloneqq q^{-1}(U)$ and $r$ an $H$-torsor. Both torsors are crepant by \cref{lem:trace-generates} over~$k$. 
Moreover,
	$\omega_{Y_U/U}\isom r^{*}\omega_{Z_U/U}\otimes\omega_{Y_U/Z_U}$\,; see \S \ref{sec:upper_shriek}. Hence
	the sheaf $\cL\coloneqq\omega_{Z_U/U}$ satisfies $r^{*}\cL\isom\cO_{Y_U}$. So $\cL$ is invertible
	 \cite[\href{https://stacks.math.columbia.edu/tag/05B2}{Tag~05B2}]{stacks-project} since
	$r$ is fppf. 
Since $\pi_*\cO_Y$ is reflexive, $H^0(Y_U,\cO_{Y_U})=H^0(Y,\cO_Y)=k$. Hence descent data on~$\cO_{Y_U}$ along the $H$-torsor~$r$ correspond to characters of~$H$, and are thus trivial. As $r^{*}\cL\isom\cO_{Y_U}$, we get $\omega_{Z_U/U}=\cL\isom\cO_{Z_U}$.

	Note that $q_*\omega_{Z/X}=\cHom_{\cO_X}(q_*\cO_Z,\cO_X)$ is reflexive on the normal scheme $X$. Hence a
	generator of $\omega_{Z_U/U}=\omega_{Z/X}|_{Z_U}$ extends uniquely to a section $\theta$ of
	$\omega_{Z/X}$.
	Since $q$ is a cover by \cref{lem:crepancy}, $q_*\cO_Z$ is reflexive. So the $\cO_Z$-linear map $\cO_Z\to\omega_{Z/X}$,
	$1\mapsto\theta$, is a morphism of reflexive $\cO_X$-modules. It is an isomorphism over the big open~$U$, hence an
	isomorphism. Thus $\omega_{Z/X}\isom\cO_Z$ in $\Coh(Z)$. Finally, $q$ is finite surjective by \cref{lem:crepancy} again, and $k\subseteq H^0(X,\cO_X)\subseteq H^0(Z,\cO_Z)\subseteq H^0(Y,\cO_Y)=k$, so \cref{lem:lift} applies.
\end{proof}

We can now use intermediate quotients to provide an alternative proof of the linear reductivity assertion of \cref{intro:tame} in the special case $R=k$.
\begin{proof}[Proof of the linear reductivity assertion of \cref{intro:tame} in the case $R=k$]
	Linear reductivity is insensitive to extension of the base field, e.g.\ by \cref{lem:group_mult}. By definition of a cover, $X$ is integral and normal, so that
	$A\coloneqq H^0(X,\cO_X)$ is a domain\,; let $K$ be its field of fractions.
Base change along $A\to K$ gives
	a quasi-torsor cover $\pi_K\colon Y_K\to X_K$ under $G_K$.
	The scheme $X_K$ is Noetherian and it is a splinter. Indeed, it is the limit of the dense open subschemes $X_f=\{f\neq0\}$ of~$X$, $f\in A\setminus\{0\}$, which are splinters (\S\ref{sec:splinters}), so that a finite surjective morphism to $X_K$ descends to a finite surjective morphism to some $X_f$ \cite[\href{https://stacks.math.columbia.edu/tag/01ZM}{Tags~01ZM}, \href{https://stacks.math.columbia.edu/tag/01ZO}{01ZO} \& \href{https://stacks.math.columbia.edu/tag/07RR}{07RR}]{stacks-project}, and a splitting over $X_f$ restricts to one over~$X_K$ by flat base change.
	By flat base change,
	$H^0(X_K,\cO_{X_K})=K=H^0(Y_K,\cO_{Y_K})$. Replacing $k$ by~$K$, we may thus assume that $H^0(X,\cO_X)=H^0(Y,\cO_Y)=k$, and $Y$ is then a
	splinter by \cref{prop:qtorsor-splinter}, in particular normal. This situation is preserved by
	base change to a finite extension $k'$ of~$k$. Indeed, the schemes $X_{k'}$ and $Y_{k'}$ are splinters \cite[Prop.~5.14]{krah-vial}, in particular normal, and
	integral, since $H^0(X_{k'},\cO_{X_{k'}})=H^0(Y_{k'},\cO_{Y_{k'}})=k'$. 
	Moreover, $\pi_{k'}$ is a
	quasi-torsor cover under~$G_{k'}$.
	 We may thus replace $k$ by any finite extension.

	By \cref{lem:group_mult}, we have to show that $G_{\bar k}$ contains no copy of $\alpha_p$ and
	that $p\nmid|\pi_0(G)|$. Suppose for contradiction that $G_{\bar k}$ contains a copy
	of~$\alpha_p$. 
	Replacing $k$ by a finite field extension, we may assume that $G$ contains a
	copy of~$\alpha_p$. By
	\cref{lem:crepancy}, $\pi$ factors as $Y\xrightarrow{\,r\,}Z=Y/\alpha_p\xrightarrow{\,q\,}X$, and
	$r$ is a quasi-torsor cover under~$\alpha_p$ since $Y$ is integral and normal. Moreover, $Z$ is a splinter by
	\cref{lem:crepant-nochar}, with $H^0(Z,\cO_Z)=k$ since $k\subseteq H^0(Z,\cO_Z)\subseteq H^0(Y,\cO_Y)=k$.
	This contradicts \cref{lem:Z_p_alpha_p}.

	It remains to prove $p\nmid|\pi_0(G)|$. By \cref{lem:crepancy} applied to
	$G^\circ\trianglelefteq G$, the scheme $Z\coloneqq Y/G^\circ$ is normal and $q\colon Z\to X$ is a
	quasi-torsor cover under $G/G^\circ$ with
	$H^0(Z,\cO_Z)=k$. By \cref{prop:qtorsor-splinter}, $Z$ is then a splinter. 
	Replacing $(Y,G)$ by
	$(Z,G/G^\circ)$, we may and do assume that $G$ is \'etale, so that $|G|=|\pi_0(G)|$.
	Suppose for contradiction
	that $p$ divides~$|G|$. Replacing $k$ by a finite separable extension, we may assume that $G$
	is a finite constant group. By Cauchy's theorem, $G$ contains a subgroup isomorphic
	to~$\ZZ/p$, and, exactly as for $\alpha_p$, Lemmas~\ref{lem:crepancy} and~\ref{lem:crepant-nochar} yield a
	factorization $\pi\colon Y\xrightarrow{\,r\,}Z=Y/(\ZZ/p)\xrightarrow{\,q\,}X$ with $r$ a
	quasi-torsor cover under $\ZZ/p$ and $Z$ a splinter with $H^0(Z,\cO_Z)=k$, contradicting
	\cref{lem:Z_p_alpha_p}. Therefore $p\nmid|G|$. By \cref{lem:group_mult}, the original $G$ is thus
	linearly reductive over the original~$k$, and hence so is $G\times_kX$ over~$X$ by
	\cite[Prop.~2.4(a)]{abramovich_olsson_vistoli_tame_stacks}.
\end{proof}

\printbibliography[%
title=References, ]

\end{document}